\documentclass{svproc}
\usepackage{amsmath, verbatim, amsfonts, amssymb, color}
\usepackage[mathscr]{euscript}
\usepackage{dsfont}

\usepackage{bm}
\usepackage{enumitem}
\usepackage{graphicx}
\definecolor{darkgreen}{rgb}{0,0.75,0}
\definecolor{darkred}{rgb}{0.75,0,0}
\definecolor{darkmagenta}{rgb}{0.5,0,0.5}
\usepackage{hyperref}
\hypersetup{colorlinks=true,citecolor=darkgreen,linkcolor=darkred,urlcolor=darkmagenta}
\usepackage{url}

\spnewtheorem{thm}{Theorem}[section]{\bfseries}{\itshape}
\spnewtheorem{conj}[thm]{Conjecture}{\bfseries}{\itshape}
\spnewtheorem{cor}[thm]{Corollary}{\bfseries}{\itshape}
\spnewtheorem{lem}[thm]{Lemma}{\bfseries}{\itshape}
\spnewtheorem{prop}[thm]{Proposition}{\bfseries}{\itshape}
\spnewtheorem{dfn}[thm]{Definition}{\bfseries}{\rmfamily}
\spnewtheorem{ass}[thm]{Assumption}{\bfseries}{\rmfamily}
\spnewtheorem{exmp}[thm]{Example}{\bfseries}{\rmfamily}
\spnewtheorem{prb}[thm]{Problem}{\bfseries}{\rmfamily}
\spnewtheorem{rmk}[thm]{Remark}{\bfseries}{\rmfamily}
\spnewtheorem*{ntn}{Notation}{\bfseries}{\rmfamily}
\spnewtheorem*{prfsktch}{Sketch of the proof}{\itshape}{\rmfamily}

\numberwithin{equation}{section}
\renewcommand{\figurename}{Figure\hspace*{1pt}}

\makeatletter
\def\@makefnmark{\hbox{(\@textsuperscript{\normalfont\@thefnmark})}}
\makeatother

\makeatletter
\@namedef{subjclassname@2020}{%
  \textup{2020} Mathematics Subject Classification}
\makeatother

\newcommand{\mr}[1]{\texttt{\href{http://www.ams.org/mathscinet-getitem?mr=#1}{MR#1}}}
\newcommand{\arxiv}[1]{\texttt{\href{http://arxiv.org/abs/#1}{arXiv:#1}}}
\newcommand{\doi}[1]{\texttt{\href{https://doi.org/#1}{doi:#1}}}
\newcommand{\ssstr}{\mathcal{L}}
\newcommand{\ssset}{K}
\newcommand{\ssindex}{S}
\newcommand{\ssproj}{\pi}
\newcommand{\words}{W}
\newcommand{\ssmap}{F}
\newcommand{\sseuc}{f}
\newcommand{\sscndc}{\rho}
\newcommand{\ssvertices}{V}
\newcommand{\ssnbd}{U}
\newcommand{\sssym}{\mathcal{G}_{\mathrm{sym}}}
\newcommand{\eucrefl}[1]{g_{#1}}
\newcommand{\SGvertex}{q}
\newcommand{\simplex}{\triangle}

\newcommand{\refmet}{d}

\newcommand{\form}{\mathcal{E}}
\newcommand{\domain}{\mathcal{F}}
\newcommand{\harfunc}[2]{\mathcal{H}_{#1,#2}}
\newcommand{\harext}[2]{h^{#1}_{#2}}
\newcommand{\harprincipal}[2]{h_{#1,#2}}
\newcommand{\harsecond}[2]{\check{h}_{#1,#2}}
\newcommand{\formsecond}{\tilde{\form}}
\newcommand{\domainsecond}{\tilde{\domain}}
\newcommand{\resismet}{R}
\newcommand{\renom}{\mathcal{R}}
\newcommand{\enermeas}[2]{\mu^{#1}_{\langle #2\rangle}}
\newcommand{\enermeasshift}[2]{\mathfrak{m}^{#1}_{\langle #2\rangle}}
\newcommand{\contfunc}{C}
\newcommand{\Borel}{\mathscr{B}}
\newcommand{\probspsigmaalg}{\mathscr{F}}
\newcommand{\formhoelderexp}{\alpha}
\newcommand{\one}{\mathds{1}} 
\newcommand{\sgn}{\mathop{\operatorname{sgn}}}
\newcommand{\osc}{\mathop{\operatorname{osc}}\nolimits}

\newcommand{\id}{\mathop{\operatorname{id}}}

\newcommand{\supp}{\mathop{\operatorname{supp}}}

\renewcommand{\ackname}{Acknowledgements.}%
\providecommand{\ack}[1]{\par\addvspace\baselineskip
\noindent\ackname\enspace\ignorespaces#1}%
\def\subjclassname{\textup{2020} \textit{Mathematics Subject Classification:}}%
\providecommand{\subjclass}[1]{\par\addvspace\baselineskip
\noindent\subjclassname\enspace\ignorespaces#1}%

\begin{document}
\mainmatter
\title{\bfseries $p$-Energy forms on fractals: recent progress}
\titlerunning{$p$-Energy forms on fractals: recent progress}

\author{\begin{picture}(0,0)\put(-85,61){Version of December 14, 2023}\end{picture}Naotaka Kajino\inst{1} \and Ryosuke Shimizu\inst{2,3}}
\authorrunning{N.~Kajino and R.~Shimizu}
\tocauthor{Naotaka Kajino and Ryosuke Shimizu}
\institute{Research Institute for Mathematical Sciences, Kyoto University, Kyoto, Japan\\
\email{nkajino@kurims.kyoto-u.ac.jp}\and
Graduate School of Informatics, Kyoto University, Kyoto, Japan\and
Department of Mathematics, Faculty of Science and Engineering, Waseda University, Tokyo, Japan (current address)\\
\email{r-shimizu@aoni.waseda.jp}}

\maketitle

\begin{abstract}
In this article, we survey recent progress on self-similar $p$-energy forms
on self-similar fractals, where $p\in(1,\infty)$. While for $p=2$ the notion of
such forms coincides with that of self-similar Dirichlet forms and there have been
plenty of studies on them since the late 1980s, studies on the case of
$p\in(1,\infty)\setminus\{2\}$ was initiated much later in 2004 by Herman,
Peirone and Strichartz [\emph{Potential Anal.}\ \textbf{20} (2004), 125--148] and
Strichartz and Wong [\emph{Nonlinearity} \textbf{17} (2004), 595--616] and
no essential progress on this case had been made since then until a few years ago.
The recent progress by Kigami, Shimizu, Cao--Gu--Qiu and Murugan--Shimizu has established
the existence of such $p$-energy forms on general post-critically finite (p.-c.f.)\ self-similar
sets and on large classes of low-dimensional infinitely ramified self-similar sets, and
the authors have proved further detailed properties of these forms and associated $p$-harmonic
functions, mainly for p.-c.f.\ self-similar sets. This article
is devoted to a review of these results, focusing on the most recent developments
by the authors and illustrating them in the simplest non-trivial setting of
the two-dimensional standard Sierpi\'{n}ski gasket.
\keywords{(Two-dimensional standard) Sierpi\'{n}ski gasket, post-critically finite (p.-c.f.)\ self-similar set,
$p$-resistance form, generalized contraction property, $p$-harmonic function, strong comparison principle, $p$-energy measure}
\subjclass{Primary 28A80, 31C45, 31E05, 60G30; secondary 31C25, 46E36}
\ack{Naotaka Kajino was supported in part by JSPS KAKENHI Grant Numbers JP21H00989, JP22H01128.
Ryosuke Shimizu was supported in part by JSPS KAKENHI Grant Numbers JP20J20207, JP23KJ2011.
This work was supported by the Research Institute for Mathematical Sciences,
an International Joint Usage/Research Center located in Kyoto University.}
\end{abstract}
\section{Introduction}\label{sec:intro}
In this article, we survey recent progress on \emph{self-similar $p$-energy forms}
on self-similar fractals for $p\in(1,\infty)$. Namely, we are concerned with a functional
$\form_{p}\colon\domain_{p}\to[0,\infty)$ defined on a linear space $\domain_{p}$ of
$\mathbb{R}$-valued functions over a fractal $\ssset$ which is self-similar with
respect to a finite family of continuous injections $\{\ssmap_{i}\}_{i\in\ssindex}$,
such that $\form_{p}^{1/p}$ is a seminorm on $\domain_{p}$ and (at least) the following hold:
\begin{gather}\label{eq:form-unit-contraction-intro}
u^{+}\wedge 1\in\domain_{p}\quad\textrm{and}\quad\form_{p}(u^{+}\wedge 1)\leq\form_{p}(u)\quad\textrm{for any $u\in\domain_{p}$,}\\
\{u\circ\ssmap_{i}\}_{i\in S}\subset\domain_{p}\quad\textrm{and}\quad
	\form_{p}(u)=\sum_{i\in\ssindex}\sscndc_{p,i}\form_{p}(u\circ\ssmap_{i})
	\quad\textrm{for any $u\in\domain_{p}$}
\label{eq:ssform-intro}
\end{gather}
for some $(\sscndc_{p,i})_{i\in\ssindex}\in(0,\infty)^{\ssindex}$. We usually assume
also that $\domain_{p}$ is complete under a suitable norm involving $\form_{p}$ and
that $\form_{p}$ satisfies some $L^{p}$-type convexity and smoothness inequalities
like $p$-Clarkson's ones. Such $p$-energy forms $(\form_{p},\domain_{p})$ could
be considered as natural analogs of the $p$-th power of the canonical seminorm
$(\int_{\mathbb{R}^{n}}\lvert\nabla u\rvert^{p}\,dx)^{1/p}$ on the first-order
$L^{p}$-Sobolev space $W^{1,p}(\mathbb{R}^{n})$ over $\mathbb{R}^{n}$.

The case of $p=2$, where such forms are nothing but self-similar Dirichlet forms on
self-similar fractals, has been extensively studied since the late 1980s (see, e.g.,
the textbooks \cite{Bar98,Kig01,Str}, a recent study \cite[Section 6]{KM23} and
references therein). On the other hand, there had been no study on the case of
$p\in(1,\infty)\setminus\{2\}$ until the first construction of such forms on a class
of post-critically finite (p.-c.f.)\ self-similar sets by Herman, Peirone and Strichartz
\cite{HPS} in 2004 and a subsequent study of the standard $p$-Laplacian on the
two-dimensional standard Sierpi\'{n}ski gasket (Figure \ref{fig:SG} below)
by Strichartz and Wong \cite{SW} in 2004, and no essential progress on this case
had been made since then until a few years ago.

The situation changed in 2021, when Kigami \cite{Kig23} and Shimizu \cite{Shi}
started the construction of self-similar $p$-energy forms on much larger classes
of self-similar fractals including the Sierpi\'{n}ski carpet and various other
infinitely ramified self-similar sets, with the motivation toward better
understanding of the Ahlfors regular conformal dimension and related
conformal-geometric properties of the Sierpi\'{n}ski carpet; see also
\cite[(1.6) and Problem 7.7]{KM23} for some backgrounds on this motivation.
Soon after the first appearance of the works \cite{Kig23,Shi} on arXiv,
Cao, Gu and Qiu \cite{CGQ} gave a comprehensive study of the case of
p.-c.f.\ self-similar sets and in particular extended the construction of
self-similar $p$-energy forms as in \cite{HPS} to general p.-c.f.\ self-similar sets.
The constructions in \cite{Kig23,Shi} needed to assume that $p$ is greater than the
Ahlfors-regular conformal dimension of the fractal $\ssset$, in order for the domain
$\domain_{p}$ to be included in the space $\contfunc(\ssset)$ of continuous functions
on $\ssset$ (see also \cite{CCK} in this connection), but Murugan and Shimizu \cite{MS}
have recently removed this restriction on $p$ in the case of the Sierpi\'{n}ski carpet.
Another important aspect of the works \cite{Shi,MS} is that they have constructed
the associated \emph{$p$-energy measure} $\enermeas{p}{u}$ of each $u\in\domain_{p}$,
i.e., the unique Borel measure $\enermeas{p}{u}$ on the fractal $\ssset$ such that
\begin{equation}\label{eq:enermeas-intro}
\enermeas{p}{u}(\ssmap_{w}(\ssset))=\sscndc_{p,w}\form_{p}(u\circ\ssmap_{w})
	\qquad\textrm{for any $w\in\bigcup_{n\in\mathbb{N}\cup\{0\}}\ssindex^{n}$,}
\end{equation}
where $\ssmap_{w}:=\ssmap_{w_{1}}\circ\cdots\circ\ssmap_{w_{n}}$ and
$\sscndc_{p,w}:=\sscndc_{p,w_{1}}\cdots\sscndc_{p,w_{n}}$ for $w=w_{1}\ldots w_{n}\in\ssindex^{n}$
($\ssmap_{w}:=\id_{\ssset}$ and $\sscndc_{p,w}:=1$ if $n=0$),
and proved some fundamental properties of this family of measures,
including the following chain rule:
\begin{equation}\label{eq:chain-rule-intro}
d\enermeas{p}{\Phi(u)}=\lvert\Phi'(u)\rvert^{p}\,d\enermeas{p}{u}
	\quad\textrm{for any $u\in\domain_{p}\cap\contfunc(\ssset)$ and any $\Phi\in\contfunc^{1}(\mathbb{R})$.}
\end{equation}

While these results have established the existence of self-similar $p$-energy forms
on very large classes of self-similar sets, detailed properties of such forms are
yet to be studied. In fact, we have recently made essential progress in analyzing
$p$-harmonic functions and $p$-energy measures $\enermeas{p}{u}$ further, mainly in the
setting of p.-c.f.\ self-similar sets. Our main results can be summarized as follows;
see \cite[Chapter 1]{Kig01} for the definition of the notion of p.-c.f.\ self-similar structure.
\begin{enumerate}[label=\textup{(\arabic*)},align=left,leftmargin=*,topsep=6pt,itemsep=2pt]
\item\label{it:GCDiff-intro}(\cite{KS:GCDiff}) The constructions of a self-similar
	$p$-energy form as in the main results of \cite{HPS,Kig23,Shi,CGQ,MS}
	can be modified so as to guarantee that the resulting $p$-energy form
	$(\form_{p},\domain_{p})$ satisfies additionally the following properties:
	\begin{enumerate}[label=\textup{(\arabic{enumi}-\arabic*)},align=left,leftmargin=*,topsep=4pt,itemsep=2pt]
	\item\label{it:GC-intro}(Generalized contraction property) If $q\in(0,p]$, $r\in[p,\infty]$,
		$m,n\in\mathbb{N}$ and $T=(T_{1},\ldots,T_{n})\colon\mathbb{R}^{m}\to\mathbb{R}^{n}$
		satisfies $\lVert T(x)-T(y)\rVert_{\ell^{r}}\leq\lVert x-y\rVert_{\ell^{q}}$
		for any $x,y\in\mathbb{R}^{m}$ and $T(0)=0$,
		then for any $u=(u_{1},\ldots,u_{m})\in\domain_{p}^{m}$,
		$T(u)=(T_{1}(u),\ldots,T_{n}(u))\in\domain_{p}^{n}$ and
		\begin{equation}\label{eq:GC-intro}
		\lVert(\form_{p}(T_{k}(u))^{1/p})_{k=1}^{n}\rVert_{\ell^{r}}
			\leq\lVert(\form_{p}(u_{j})^{1/p})_{j=1}^{m}\rVert_{\ell^{q}}.
		\end{equation}
	\item\label{it:Diff-intro}(Differentiability) $\form_{p}$ is Fr\'{e}chet differentiable
		with locally $\formhoelderexp_{p}$-H\"{o}lder continuous Fr\'{e}chet derivative with respect
		to a natural norm on $\domain_{p}$, where $\formhoelderexp_{p}:=\frac{1}{p}\wedge\frac{p-1}{p}$.
	\end{enumerate}
\item\label{it:StrComp-pqEnergySing-intro}Assume that $\ssstr:=(\ssset,\ssindex,\{\ssmap_{i}\}_{i\in\ssindex})$ is a
	p.-c.f.\ self-similar structure with $\ssset$ connected and containing at least
	two elements, and that $(\form_{p},\domain_{p})$ is a self-similar $p$-energy
	form over $\ssstr$ satisfying \ref{it:GC-intro}. Then the following hold:
	\begin{enumerate}[label=\textup{(\arabic{enumi}-\arabic*)},align=left,leftmargin=*,topsep=4pt,itemsep=2pt]
	\item\label{it:StrComp-intro}(Strong comparison principle; \cite{KS:StrComp}) Let $U\subsetneq\ssset$ be a
		connected open subset of $\ssset$. If $u,v\in\contfunc(\overline{U}^{\ssset})$
		are harmonic on $U$ with respect to $\form_{p}$ and $u(x)\leq v(x)$ for any $x\in\partial_{\ssset}U$,
		then either $u(x)<v(x)$ for any $x\in U$ or $u(x)=v(x)$ for any $x\in\overline{U}^{\ssset}$.
	\item\label{it:p-RF-unique-intro}(Uniqueness under good symmetry; \cite{KS:StrComp}) If $\ssstr=(\ssset,\ssindex,\{\ssmap_{i}\}_{i\in\ssindex})$
		has certain good geometric symmetry and $(\form_{p},\domain_{p})$ is invariant
		under this symmetry of $\ssstr$, then any other such $p$-energy form over $\ssset$
		satisfying \eqref{eq:ssform-intro} with the same $(\sscndc_{p,i})_{i\in\ssindex}$ as
		$(\form_{p},\domain_{p})$ has domain $\domain_{p}$ and is a constant multiple of $\form_{p}$.
	\item\label{it:pqEnergySing-intro}(Singularity of $p$- and $q$-energy measures under very good symmetry; \cite{KS:pqEnergySing}) If $\ssstr=(\ssset,\ssindex,\{\ssmap_{i}\}_{i\in\ssindex})$
		has certain very good geometric symmetry, $p,q\in(1,\infty)$, $p\not=q$, and
		$(\form_{p},\domain_{p})$ and $(\form_{q},\domain_{q})$ are invariant
		under this symmetry of $\ssstr$, then the $p$-energy measure $\enermeas{p}{u}$
		of any $u\in\domain_{p}$ and the $q$-energy measure $\enermeas{q}{v}$ of
		any $v\in\domain_{q}$ are mutually singular.
	\end{enumerate}
\end{enumerate}

The purpose of this article is to give some concise descriptions of how these
results are related to each other and how they can be obtained, by illustrating them
in the simplest non-trivial setting of the (two-dimensional standard) Sierpi\'{n}ski
gasket (Figure \ref{fig:SG} below), even for which all the results in
\ref{it:GCDiff-intro} and \ref{it:StrComp-pqEnergySing-intro} above are new.
In fact, Strichartz and Wong \cite{SW} supposed without proofs some conditions implied
by \ref{it:GC-intro}, \ref{it:Diff-intro} and \ref{it:StrComp-intro} on the canonical
$p$-energy form on the Sierpi\'{n}ski gasket to derive some properties of it.
Our above results \ref{it:GCDiff-intro} and \ref{it:StrComp-intro} ensure that
a $p$-energy form on the Sierpi\'{n}ski gasket satisfying those conditions exists
and thereby turns out to have the properties derived in \cite{SW};
see also Subsection \ref{ssec:SG-strong-comp-p-RF-unique} below, where we provide
a self-contained treatment of principal results in \cite[Section 5]{SW}.

The rest of this paper is organized as follows. First in Section \ref{sec:p-RF},
we give a brief summary of the theory of $p$-resistance forms planned to appear
in \cite{KS:GCDiff}. In Section \ref{sec:SG-p-RF}, after introducing the
(two-dimensional standard) Sierpi\'{n}ski gasket in Subsection \ref{ssec:SG},
we survey the construction and basic properties of the canonical $p$-resistance form
on the Sierpi\'{n}ski gasket in Subsection \ref{ssec:SG-construct-p-RF}, describing the
modification obtained in \cite{KS:GCDiff} of the original construction due to \cite{HPS,CGQ}.
In Subsection \ref{ssec:SG-strong-comp-p-RF-unique}, we present some details of the
strong comparison principle and its application to uniqueness of a self-similar
$p$-resistance form planned to be presented in \cite{KS:StrComp}. Finally in
Section \ref{sec:SG-p-energy-meas}, keeping the setting of the Sierpi\'{n}ski gasket,
we introduce the $p$-energy measures, mention their basic properties, and prove that
the $p$-energy measures and the $q$-energy measures are mutually singular for any
$p,q\in(1,\infty)$ with $p\not=q$, whose proof in the greater generality described
in \ref{it:pqEnergySing-intro} above is planned to appear in \cite{KS:pqEnergySing}.
\begin{ntn}
In this paper, we adopt the following notation and conventions.
\begin{enumerate}[label=\textup{(\arabic*)},align=left,leftmargin=*,topsep=4pt,itemsep=2pt]
\item The symbols $\subset$ and $\supset$ for set inclusion
	\emph{allow} the case of the equality.
\item $\mathbb{N}:=\{n\in\mathbb{Z}\mid n>0\}$, i.e., $0\not\in\mathbb{N}$.
\item The cardinality (the number of elements) of a set $A$ is denoted by $\#A$.
\item We set $\sup\emptyset:=0$, set $a\vee b:=\max\{a,b\}$,
	$a\wedge b:=\min\{a,b\}$, $a^{+}:=a\vee 0$, $a^{-}:=-(a\wedge 0)$ and
	$\sgn(a):=\lvert a\rvert^{-1}a$ ($\sgn(0):=0$) for $a,b\in\mathbb{R}$,
	and use the same notation also for $\mathbb{R}$-valued
	functions and equivalence classes of them. All numerical functions
	in this paper are assumed to be $[-\infty,\infty]$-valued.
\item Let $n\in\mathbb{N}$. For $x=(x_{k})_{k=1}^{n}\in\mathbb{R}^{n}$, we set
	$\lVert x\rVert_{\ell^{p}}:=(\sum_{k=1}^{n}\lvert x_{k}\rvert^{p})^{1/p}$
	for $p\in(0,\infty)$, $\lVert x\rVert_{\ell^{\infty}}:=\max_{1\leq k\leq n}\lvert x_{k}\rvert$
	and $\lvert x\rvert:=\lVert x\rVert_{\ell^{2}}$. For $\Phi\colon\mathbb{R}^{n}\to\mathbb{R}$
	which is differentiable on $\mathbb{R}^{n}$ and for $k\in\{1,\ldots,n\}$, its first-order
	partial derivative in the $k$-th coordinate is denoted by $\partial_{k}\Phi$, and we set
	$\contfunc^{1}(\mathbb{R}^{n}):=\{\Phi\mid\textrm{$\Phi\colon\mathbb{R}^{n}\to\mathbb{R}$, $\Phi$ is
	differentiable on $\mathbb{R}^{n}$, $\partial_{1}\Phi,\ldots,\partial_{n}\Phi$ are continuous on $\mathbb{R}^{n}$}\}$.
\item Let $K$ be a non-empty set. We define $\id_{K}\colon K\to K$ by $\id_{K}(x):=x$,
	$\one_{A}=\one^{K}_{A}\in\mathbb{R}^{K}$ for $A\subset K$ by
	$\one_{A}(x):=\one^{K}_{A}(x):=\bigl\{\begin{smallmatrix}1 & \textrm{if $x\in A$,}\\ 0 & \textrm{if $x\not\in A$,}\end{smallmatrix}$
	set $\one_{x}:=\one^{K}_{x}:=\one_{\{x\}}$ for $x\in K$,
	$\lVert u\rVert_{\sup}:=\lVert u\rVert_{\sup,K}:=\sup_{x\in K}\lvert u(x)\rvert$
	for $u\in\mathbb{R}^{K}$, and
	$\osc_{K}[u]:=\sup_{x,y\in K}\lvert u(x)-u(y)\rvert=\sup_{x\in K}u(x)-\inf_{x\in K}u(x)$
	for $u\in\mathbb{R}^{K}$ with $\lVert u\rVert_{\sup}<\infty$.
\item Let $K$ be a topological space. The Borel $\sigma$-algebra of $K$ is denoted by $\Borel(K)$.
	The closure and boundary of $A\subset K$ in $K$ are denoted by
	$\overline{A}^{K}$ and $\partial_{K}A$, respectively.
	We set $\contfunc(K):=\{u\in\mathbb{R}^{K}\mid\textrm{$u$ is continuous}\}$ and
	$\supp_{K}[u]:=\overline{K\setminus u^{-1}(0)}^{K}$ for $u\in\contfunc(K)$.
\item Let $(K,\mathscr{B})$ be a measurable space and let $\mu,\nu$ be measures on
	$(K,\mathscr{B})$. We write $\nu \ll \mu$ and $\nu \perp \mu$ to mean that
	$\nu$ is absolutely continuous and singular, respectively, with respect to $\mu$.
\end{enumerate}
\end{ntn}
%
\section{$p$-Resistance forms: the definition and basic properties}\label{sec:p-RF}
As treated in \cite{Bar98,Kig01,Str}, the theory of self-similar Dirichlet forms on
p.-c.f.\ self-similar sets has been developed on the basis of that of resistance forms.
Likewise, the construction and basic properties of self-similar $p$-energy forms on
p.-c.f.\ self-similar sets can be presented most efficiently by referring to some
general facts for certain natural $L^{p}$-analogs of resistance forms which we call
\emph{$p$-resistance forms}. For this reason, we start our discussion with introducing
the notion of such forms and stating some basic properties of them. The details
of the results in this section will appear in \cite{KS:GCDiff}.

Let $p\in(1,\infty)$, which we fix throughout this section.
\begin{dfn}[$p$-Resistance forms]\label{dfn:p-RF}
Let $\ssset$ be a non-empty set. The pair $(\form_{p},\domain_{p})$ of a linear subspace
$\domain_{p}$ of $\mathbb{R}^{\ssset}$ and a functional $\form_{p}\colon\domain_{p}\to[0,\infty)$
is said to be a \emph{$p$-resistance form} on $\ssset$ if and only if the following
five conditions are satisfied:
\begin{enumerate}[label=\textup{(RF\arabic*)$_{p}$},align=left,leftmargin=*,topsep=4pt,itemsep=2pt]
\hypertarget{p-RF1}{\item}\label{it:p-RF1}$\form_{p}^{1/p}$ is a seminorm on $\domain_{p}$ and
	$\{u\in\domain_{p}\mid\form_{p}(u)=0\}=\mathbb{R}\one_{\ssset}$.
\hypertarget{p-RF2}{\item}\label{it:p-RF2}The quotient normed space $(\domain_{p}/\mathbb{R}\one_{\ssset},\form_{p}^{1/p})$ is a Banach space.
\hypertarget{p-RF3}{\item}\label{it:p-RF3}For any $x,y\in\ssset$ with $x\not=y$ there exists $u\in\domain_{p}$ such that $u(x)\not=u(y)$.
\hypertarget{p-RF4}{\item}\label{it:p-RF4}$R_{\form_{p}}(x,y):=\sup\biggl\{\dfrac{\lvert u(x)-u(y)\rvert^{p}}{\form_{p}(u)}\biggm\vert u\in\domain_{p}\setminus\mathbb{R}\one_{\ssset}\biggr\}<\infty$
	for any $x,y\in\ssset$.
\hypertarget{p-RF5}{\item}\label{it:p-RF5}(Generalized contraction property) If $q\in(0,p]$, $r\in[p,\infty]$,
	$m,n\in\mathbb{N}$ and $T=(T_{1},\ldots,T_{n})\colon\mathbb{R}^{m}\to\mathbb{R}^{n}$
	satisfies $\lVert T(x)-T(y)\rVert_{\ell^{r}}\leq\lVert x-y\rVert_{\ell^{q}}$
	for any $x,y\in\mathbb{R}^{m}$ and $T(0)=0$, then for any
	$u=(u_{1},\ldots,u_{m})\in\domain_{p}^{m}$, $T(u)=(T_{1}(u),\ldots,T_{n}(u))\in\domain_{p}^{n}$ and
	\begin{equation}\label{eq:GC}
	\lVert(\form_{p}(T_{k}(u))^{1/p})_{k=1}^{n}\rVert_{\ell^{r}}
		\leq\lVert(\form_{p}(u_{j})^{1/p})_{j=1}^{m}\rVert_{\ell^{q}}.
	\end{equation}
\end{enumerate}
\end{dfn}

The generalized contraction property \ref{it:p-RF5} is arguably the strongest possible
form of contraction properties of energy functionals, including as the special case with
$p=2$, $q=1$ and $n=1$ the so-called normal contractivity of symmetric Dirichlet forms
(see, e.g., \cite[Theorem 4.12]{MR} or \cite[Proposition I.3.3.1]{BH}). As discussed
in Example \ref{exmp:p-RF}-\ref{it:p-RF-RF} and Remark \ref{rmk:p-RF5-DF} below,
if $\form_{2}\colon\domain_{2}\to[0,\infty)$ is a quadratic form on a linear space
$\domain_{2}$ of (equivalence classes of) $\mathbb{R}$-valued functions and has some
suitable completeness property, then \hyperlink{p-RF5}{\textup{(RF5)$_{2}$}}
is equivalent to \eqref{eq:form-unit-contraction-intro}.
On the other hand, for $p\not=2$, at the moment we do not know any characterization
of \ref{it:p-RF5} by simpler conditions, but it turns out that we can still modify the
existing constructions of self-similar $p$-energy forms in \cite{HPS,Kig23,Shi,CGQ,MS}
so as to get ones satisfying \ref{it:p-RF5}. In this sense we miss essentially no
examples by requiring our $p$-energy forms to satisfy the strong condition \ref{it:p-RF5}.
\begin{exmp}\label{exmp:p-RF}
\begin{enumerate}[label=\textup{(\arabic*)},align=left,leftmargin=*,topsep=4pt,itemsep=2pt]
\item\label{it:p-RF-finite}Let $V$ be a non-empty finite set. Note that in this case
	$(\form_{p},\domain_{p})$ is a $p$-resistance form on $V$ if and only if $\domain_{p}=\mathbb{R}^{V}$
	and $\form_{p}\colon\mathbb{R}^{V}\to[0,\infty)$ satisfies \ref{it:p-RF1} and \ref{it:p-RF5};
	indeed, since \ref{it:p-RF5} implies \eqref{eq:form-unit-contraction-intro},
	\ref{it:p-RF3} is equivalent to $\domain_{p}=\mathbb{R}^{V}$ under \ref{it:p-RF1} and
	\ref{it:p-RF5} by \cite[Proposition 3.2]{Kig12}, and \ref{it:p-RF2} and \ref{it:p-RF4}
	are easily implied by \ref{it:p-RF1} and $\dim\domain_{p}/\mathbb{R}\one_{V}<\infty$.
	On the basis of this observation, for $p$-resistance forms $(\form_{p},\domain_{p})$
	on finite sets, we refer only to $\form_{p}$ and say simply that
	$\form_{p}$ is a $p$-resistance form on the set.
	
	Now, consider any functional $\form_{p}\colon\mathbb{R}^{V}\to[0,\infty)$ of the form
	\begin{equation}\label{eq:p-RF-finite-graph}
	\form_{p}(u)=\frac{1}{2}\sum_{x,y\in V}L_{xy}\lvert u(x)-u(y)\rvert^{p}
	\end{equation}
	for some $L=(L_{xy})_{x,y\in V}\in[0,\infty)^{V\times V}$ such that $L_{xy}=L_{yx}$ for any $x,y\in V$.
	It is obvious that $\form_{p}$ satisfies \ref{it:p-RF1}
	if and only if the graph $(V,E_{L})$ is connected, where
	$E_{L}:=\{\{x,y\}\mid\textrm{$x,y\in V$, $x\not=y$, $L_{xy}>0$}\}$.
	It is also easy to see, by using the reverse Minkowski inequality for
	$\lVert\cdot\rVert_{\ell^{p/r}}$ (see, e.g., \cite[Theorem 2.13]{AF})
	and Minkowski's inequality for $\lVert\cdot\rVert_{\ell^{p/q}}$, that
	$\form_{p}$ satisfies \ref{it:p-RF5}. It thus follows that $\form_{p}$
	is a $p$-resistance form on $V$ if and only if $(V,E_{L})$ is connected.
	
	Note that, while any $2$-resistance form on $V$ is of the form \eqref{eq:p-RF-finite-graph}
	with $p=2$, \emph{the counterpart of this fact for $p\not=2$ is NOT true} unless $\#V\leq 2$.
	As we will see in Section \ref{sec:SG-p-RF} and summarize in
	Remark \ref{rmk:SG-p-RF-renom-trace-necessary}, the analysis of self-similar $p$-energy forms
	on p.-c.f.\ self-similar sets inevitably involves $p$-energy forms on finite sets
	which may not be of the form \eqref{eq:p-RF-finite-graph}, and therefore the validity
	of contraction properties like \ref{it:p-RF5} for such forms becomes non-trivial.
\item\label{it:p-RF-RF}Recall from \cite[Definition 2.3.1]{Kig01} and \cite[Definition 3.1]{Kig12}
	that a \emph{resistance form} on a non-empty set $\ssset$ is defined as the pair
	$(\form,\domain)$ of a linear subspace $\domain$ of $\mathbb{R}^{\ssset}$ and
	a non-negative definite symmetric bilinear form $\form\colon\domain\times\domain\to\mathbb{R}$
	such that the associated quadratic form $(\form_{2},\domain)$
	given by $\form_{2}(u):=\form(u,u)$ satisfies
	\hyperlink{p-RF1}{\textup{(RF1)$_{2}$}}, \hyperlink{p-RF2}{\textup{(RF2)$_{2}$}},
	\hyperlink{p-RF3}{\textup{(RF3)$_{2}$}}, \hyperlink{p-RF4}{\textup{(RF4)$_{2}$}}
	and \eqref{eq:form-unit-contraction-intro}. As our terminology indicates,
	resistance forms on $\ssset$ can be canonically identified with
	$2$-resistance forms on $\ssset$ as defined in Definition \ref{dfn:p-RF},
	through the correspondence $(\form,\domain)\mapsto(\form_{2},\domain)$.
	Indeed, if $(\form,\domain)$ is a resistance form on $\ssset$, then
	$(\form_{2},\domain)$ satisfies \hyperlink{p-RF5}{\textup{(RF5)$_{2}$}} since, by
	\cite[Corollary 2.37]{K:Fractal2018} (see also \cite[Theorems 2.3.6, 2.3.7 and Lemma 2.3.8]{Kig01}),
	$\form_{2}$ can be expressed as a certain kind of supremum over resistance forms
	on the non-empty finite subsets of $\ssset$, which are necessarily of the form
	\eqref{eq:p-RF-finite-graph} with $p=2$ and hence satisfy
	\hyperlink{p-RF5}{\textup{(RF5)$_{2}$}} by \ref{it:p-RF-finite} above.
	Conversely, if $(\form_{2},\domain_{2})$ is a $2$-resistance form on $\ssset$,
	then two simple applications of \hyperlink{p-RF5}{\textup{(RF5)$_{2}$}} show
	\eqref{eq:form-unit-contraction-intro} with $p=2$ and that
	$\form_{2}(u+v)+\form_{2}(u-v)=2(\form_{2}(u)+\form_{2}(v))$
	for any $u,v\in\domain_{2}$, so that
	$\form\colon\domain_{2}\times\domain_{2}\to\mathbb{R}$ defined by
	$\form(u,v):=\frac{1}{4}(\form_{2}(u+v)-\form_{2}(u-v))$ is bilinear and symmetric,
	satisfies $\form(u,u)=\form_{2}(u)\geq 0$ for any $u\in\domain_{2}$ and thus gives
	a resistance form $(\form,\domain_{2})$ on $\ssset$.
\end{enumerate}
\end{exmp}
\begin{rmk}\label{rmk:p-RF5-DF}
Similarly, \emph{any symmetric Dirichlet form satisfies \hyperlink{p-RF5}{\textup{(RF5)$_{2}$}}},
which does not seem to have ever been stated in the literature.
Indeed, this claim follows by using the argument in Example \ref{exmp:p-RF}-\ref{it:p-RF-finite}
based on the reverse Minkowski inequality for $\lVert\cdot\rVert_{\ell^{2/r}}$
and Minkowski's inequality for $\lVert\cdot\rVert_{\ell^{2/q}}$
to modify \cite[Proof of Theorem 4.12]{MR} in the right manner.
\end{rmk}

Another application of Minkowski's inequality for $\lVert\cdot\rVert_{\ell^{r/p}}$
and the reverse Minkowski inequality for $\lVert\cdot\rVert_{\ell^{q/p}}$
also shows the following lemma.
\begin{lem}\label{lem:p-RF1-p-RF5-cone}
Let $\ssset$ be a non-empty set, let $\domain_{p}$ be a linear subspace of $\mathbb{R}^{\ssset}$,
let $\form_{p,1},\form_{p,2}\colon\domain_{p}\to[0,\infty)$ and $a_{1},a_{2}\in[0,\infty)$.
\begin{enumerate}[label=\textup{(\arabic*)},align=left,leftmargin=*,topsep=4pt,itemsep=2pt]
\item\label{it:p-RF1-cone}If $\form_{p,1}^{1/p},\form_{p,2}^{1/p}$ are seminorms on
	$\domain_{p}$, then so is $(a_{1}\form_{p,1}+a_{2}\form_{p,2})^{1/p}$.
\item\label{it:p-RF5-cone}If $(\form_{p,1},\domain_{p}),(\form_{p,2},\domain_{p})$
	satisfy \ref{it:p-RF5}, then so does $(a_{1}\form_{p,1}+a_{2}\form_{p,2},\domain_{p})$.
\end{enumerate}
\end{lem}

Throughout the rest of this section, we assume that $\ssset$ is a non-empty set and that
$(\form_{p},\domain_{p})$ is a $p$-resistance form on $\ssset$. The following inequality,
which is immediate from \ref{it:p-RF4}, is of fundamental importance in this setting.
\begin{prop}\label{prop:p-RF4-conseq}
For any $u\in\domain_{p}$ and any $x,y\in\ssset$,
\begin{equation}\label{eq:p-RF4-conseq}
\lvert u(x)-u(y)\rvert^{p}\leq\resismet_{\form_{p}}(x,y)\form_{p}(u).
\end{equation}
\end{prop}

\ref{it:p-RF5} transfers various inequalities satisfied by $L^{p}$-norms
to $\form_{p}^{1/p}$ and allows us thereby to analyze $p$-energy forms more deeply.
We collect some important special cases of \ref{it:p-RF5} in the following proposition;
see \cite[Theorem 4.7]{Cla} and the references therein for some background
for \ref{it:p-RF-StronglySubadditiveGen} and, e.g.,
\cite[Lemma 2.37 and Theorem 2.38]{AF} for \ref{it:p-RF-Clarkson}.
\begin{prop}\label{prop:p-RF5-conseq}
Let $u,v\in\domain_{p}$, and let $\varphi\in\contfunc(\mathbb{R})$ satisfy
$\lvert\varphi(t)-\varphi(s)\rvert\leq\lvert t-s\rvert$ for any $s,t\in\mathbb{R}$.
\begin{enumerate}[label=\textup{(\arabic*)},align=left,leftmargin=*,topsep=4pt,itemsep=2pt]
\item\label{it:p-RF-1Lip}$\varphi(u)\in\domain_{p}$ and $\form_{p}(\varphi(u))\leq\form_{p}(u)$.
	In particular, $f\in\domain_{p}$ and $\form_{p}(f)\leq\form_{p}(u)$
	for any $f\in\{u^{+}\wedge 1,\lvert u\rvert,u^{+},u^{-}\}$, and
	$u-\varphi(u-v),v+\varphi(u-v),u\wedge v,u\vee v\in\domain_{p}$.
\item\label{it:p-RF-product}If $\lVert u\rVert_{\sup}\vee\lVert v\rVert_{\sup}<\infty$,
	then $uv\in\domain_{p}$ and
	$\form_{p}(uv)^{1/p}\leq\lVert v\rVert_{\sup}\form_{p}(u)^{1/p}+\lVert u\rVert_{\sup}\form_{p}(v)^{1/p}$.
\item\label{it:p-RF-StronglySubadditive}\textup{(Strong subadditivity)}
	$\form_{p}(u\wedge v)+\form_{p}(u\vee v)\leq\form_{p}(u)+\form_{p}(v)$.
\item\label{it:p-RF-StronglySubadditiveGen}If $\varphi$ is non-decreasing, then
	$\form_{p}(u-\varphi(u-v))+\form_{p}(v+\varphi(u-v))\leq\form_{p}(u)+\form_{p}(v)$.
\item\label{it:p-RF-Clarkson}\textup{($p$-Clarkson's inequalities)} If $p\leq 2$, then
	\begin{align}\label{eq:p-RF-sClarkson-pleq2}
	\form_{p}(u+v)+\form_{p}(u-v)&\geq 2\bigl(\form_{p}(u)^{1/(p-1)}+\form_{p}(v)^{1/(p-1)}\bigr)^{p-1},\\
	\form_{p}(u+v)+\form_{p}(u-v)&\leq 2(\form_{p}(u)+\form_{p}(v)).
	\label{eq:p-RF-wClarkson-pleq2}
	\end{align}
	If $p\geq 2$, then
	\begin{align}\label{eq:p-RF-sClarkson-pgeq2}
		\form_{p}(u+v)+\form_{p}(u-v)&\leq 2\bigl(\form_{p}(u)^{1/(p-1)}+\form_{p}(v)^{1/(p-1)}\bigr)^{p-1},\\
		\form_{p}(u+v)+\form_{p}(u-v)&\geq 2(\form_{p}(u)+\form_{p}(v)).
	\label{eq:p-RF-wClarkson-pgeq2}
	\end{align}
\end{enumerate}
\end{prop}

$p$-Clarkson's inequalities \eqref{eq:p-RF-wClarkson-pleq2} and \eqref{eq:p-RF-sClarkson-pgeq2},
together with \ref{it:p-RF1}, the convexity of $\lvert\cdot\rvert^{p}$,
\eqref{eq:p-RF-sClarkson-pleq2} and \eqref{eq:p-RF-wClarkson-pgeq2}, imply the following properties.
\begin{thm}\label{thm:p-RF-Diff}
The quotient norm $\form_{p}\colon\domain_{p}/\mathbb{R}\one_{\ssset}\to\mathbb{R}$
is Fr\'{e}chet differentiable on $(\domain_{p}/\mathbb{R}\one_{\ssset},\form_{p}^{1/p})$.
In particular, for any $u,v\in\domain_{p}$,
\begin{equation}\label{eq:p-RF-Diff-dfn}
\textrm{the derivative}\quad
	\form_{p}(u;v):=\frac{1}{p}\frac{d}{dt}\form_{p}(u+tv)\bigg\vert_{t=0}\in\mathbb{R}
	\quad\textrm{exists,}
\end{equation}
$\form_{p}(u;\cdot)\colon\domain_{p}\to\mathbb{R}$ is linear,
$\form_{p}(u;u)=\form_{p}(u)$ and $\form_{p}(u;\one_{\ssset})=0$. Moreover,
for any $u,u_{1},u_{2},v\in\domain_{p}$ and any $a\in\mathbb{R}$, the following hold:
\begin{gather}\label{eq:p-RF-Diff-strictly-convex}
\textrm{$\mathbb{R}\ni t\mapsto\form_{p}(u+tv;v)\in\mathbb{R}$ is strictly increasing if and only if $v\not\in\mathbb{R}\one_{\ssset}$.}\\
\form_{p}(au;v)=\sgn(a)\lvert a\rvert^{p-1}\form_{p}(u;v),\qquad
	\form_{p}(u+a\one_{\ssset};v)=\form_{p}(u;v).\label{eq:p-RF-Diff-homogeneous}\\
\lvert\form_{p}(u;v)\rvert\leq\form_{p}(u)^{(p-1)/p}\form_{p}(v)^{1/p}.\label{eq:p-RF-Diff-Hoelder}\\
\lvert\form_{p}(u_{1};v)-\form_{p}(u_{2};v)\rvert\leq c_{p}(\form_{p}(u_{1})\vee\form_{p}(u_{2}))^{(p-1-\formhoelderexp_{p})/p}\form_{p}(u_{1}-u_{2})^{\formhoelderexp_{p}/p}\form_{p}(v)^{1/p}
\label{eq:p-RF-Diff-Hoelder-cont}
\end{gather}
for $\formhoelderexp_{p}:=\frac{1}{p}\wedge\frac{p-1}{p}$ and some $c_{p}\in(0,\infty)$
determined solely and explicitly by $p$.
\end{thm}

Theorem \ref{thm:p-RF-Diff} plays fundamental roles in establishing fine properties
of harmonic functions with respect to $p$-resistance forms. Below we summarize some
important consequences of it. First, the following proposition states that the
variational and distributional formulations of the notion of harmonicity of functions
coincide for $p$-resistance forms.
\begin{prop}\label{prop:p-RF-harm-func}
Let $h\in\domain_{p}$ and $B\subset\ssset$. Then the following two conditions are equivalent:
\begin{enumerate}[label=\textup{(\arabic*)},align=left,leftmargin=*,topsep=4pt,itemsep=2pt]
\item\label{it:p-RF-harm-func-min}$\form_{p}(h)=\inf\{\form_{p}(u)\mid\textrm{$u\in\domain_{p}$, $u\vert_{B}=h\vert_{B}$}\}$.
\item\label{it:p-RF-harm-func-test}$\form_{p}(h;v)=0$ for any $v\in\domain_{p}$ with $v\vert_{B}=0$.
\end{enumerate}
\end{prop}
\begin{dfn}[$\form_{p}$-harmonic functions]\label{dfn:p-RF-harm-func}
Let $B\subset\ssset$. We say that $h\in\domain_{p}$ is \emph{$\form_{p}$-harmonic} on
$\ssset\setminus B$ if and only if $h$ satisfies either (and hence both) of
\ref{it:p-RF-harm-func-min} and \ref{it:p-RF-harm-func-test} in Proposition \ref{prop:p-RF-harm-func}.
We set $\harfunc{\form_{p}}{B}:=\{h\in\domain_{p}\mid\textrm{$h$ is $\form_{p}$-harmonic on $\ssset\setminus B$}\}$.
\end{dfn}

$\form_{p}$-harmonic functions with given boundary values uniquely exist, and their energies
under $\form_{p}$ define a new $p$-resistance form on the boundary set, as follows.%
\begin{thm}\label{thm:p-RF-trace}
Let $B\subset\ssset$ be non-empty, set $\domain_{p}\vert_{B}:=\{u\vert_{B}\mid u\in\domain_{p}\}$,
and define $\form_{p}\vert_{B}\colon\domain_{p}\vert_{B}\to[0,\infty)$ by
\begin{equation}\label{eq:p-RF-trace-form}
\form_{p}\vert_{B}(u):=\inf\{\form_{p}(v)\mid\textrm{$v\in\domain_{p}$, $v\vert_{B}=u$}\},
	\quad u\in\domain_{p}\vert_{B}.
\end{equation}
Then $(\form_{p}\vert_{B},\domain_{p}\vert_{B})$ is a $p$-resistance form on $B$ and
$\resismet_{\form_{p}\vert_{B}}=\resismet_{\form_{p}}\vert_{B\times B}$ \textup{(recall \ref{it:p-RF4})}.
Moreover, for any $u\in\domain_{p}\vert_{B}$ there exists a unique
$\harext{\form_{p}}{B}[u]\in\domain_{p}$ such that $\harext{\form_{p}}{B}[u]\big\vert_{B}=u$
and $\form_{p}\bigl(\harext{\form_{p}}{B}[u]\bigr)=\form_{p}\vert_{B}(u)$,
so that $\harfunc{\form_{p}}{B}=\harext{\form_{p}}{B}(\domain_{p}\vert_{B})$, and
\begin{gather}\label{eq:p-RF-harext-homogeneous}
\harext{\form_{p}}{B}[au+b\one_{B}]=a\harext{\form_{p}}{B}[u]+b\one_{\ssset}
	\quad\textrm{for any $u\in\domain_{p}\vert_{B}$ and any $a,b\in\mathbb{R}$,}\\
\form_{p}\vert_{B}(u\vert_{B};v\vert_{B})=\form_{p}(u;v)
	\quad\textrm{for any $u\in\harfunc{\form_{p}}{B}$ and any $v\in\domain_{p}$,}
\label{eq:p-RF-trace-Diff}
\end{gather}
where $\form_{p}\vert_{B}(u;v):=\frac{1}{p}\frac{d}{dt}\form_{p}\vert_{B}(u+tv)\vert_{t=0}$
for $u,v\in\domain_{p}\vert_{B}$ \textup{(recall \eqref{eq:p-RF-Diff-dfn})}.
\end{thm}
\begin{rmk}\label{rmk:p-RF-harext-homogeneous}
In contrast to \eqref{eq:p-RF-harext-homogeneous}, \emph{the map
$\harext{\form_{p}}{B}\colon\domain_{p}\vert_{B}\to\domain_{p}$ does NOT satisfy
either $\harext{\form_{p}}{B}[u+v]\leq\harext{\form_{p}}{B}[u]+\harext{\form_{p}}{B}[v]$
for any $u,v\in\domain_{p}\vert_{B}$ or
$\harext{\form_{p}}{B}[u+v]\geq\harext{\form_{p}}{B}[u]+\harext{\form_{p}}{B}[v]$
for any $u,v\in\domain_{p}\vert_{B}$ in general}, unless $p=2$ or $\#B\leq 2$.
\end{rmk}
\begin{dfn}[Trace of $p$-resistance form]\label{dfn:p-RF-trace}
Let $B\subset\ssset$ be non-empty. The $p$-resistance form
$(\form_{p}\vert_{B},\domain_{p}\vert_{B})$ on $B$ as defined in Theorem \ref{thm:p-RF-trace}
is called the \emph{trace} of $(\form_{p},\domain_{p})$ to $B$.
\end{dfn}
\begin{prop}\label{prop:p-RF-trace-compatible}
Let $A,B\subset\ssset$ satisfy $\emptyset\not=A\subset B$. Then
$(\form_{p}\vert_{B}\vert_{A},\domain_{p}\vert_{B}\vert_{A})=(\form_{p}\vert_{A},\domain_{p}\vert_{A})$
and $\harext{\form_{p}}{B}\circ\harext{\form_{p}\vert_{B}}{A}=\harext{\form_{p}}{A}$.
In particular, $\harext{\form_{p}\vert_{B}}{A}[u]=\harext{\form_{p}}{A}[u]\big\vert_{B}$
for any $u\in\domain_{p}\vert_{A}$.
\end{prop}

As an easy consequence of the strong subadditivity of $\form_{p}$
(Proposition \ref{prop:p-RF5-conseq}-\ref{it:p-RF-StronglySubadditive}),
we have the following natural analog of weak maximum principle.
\begin{prop}[Weak comparison principle]\label{prop:p-RF-comp}
Let $B\subset\ssset$ be non-empty, and let $u,v\in\harfunc{\form_{p}}{B}$ satisfy
$u(x)\leq v(x)$ for any $x\in B$. Then $u(x)\leq v(x)$ for any $x\in\ssset$.
\end{prop}

The strong subadditivity of $\form_{p}$
(Proposition \ref{prop:p-RF5-conseq}-\ref{it:p-RF-StronglySubadditive})
also implies the following important property of the derivative of $\form_{p}$.
\begin{prop}\label{prop:p-RF-Diff-monotone}
Let $u_{1},u_{2},v\in\domain_{p}$ satisfy $((u_{2}-u_{1})\wedge v)(x)=0$ for any $x\in\ssset$.
Then $\form_{p}(u_{1};v)\geq\form_{p}(u_{2};v)$.
\end{prop}

Combining \eqref{eq:p-RF-Diff-homogeneous}, Proposition \ref{prop:p-RF-harm-func},
Theorem \ref{thm:p-RF-trace}, Propositions \ref{prop:p-RF-trace-compatible},
\ref{prop:p-RF-comp} and \ref{prop:p-RF-Diff-monotone},
we can prove the following H\"{o}lder continuity estimate for $\form_{p}$-harmonic
functions, which plays central roles in our analysis of $p$-harmonic functions and
$p$-energy measures for self-similar $p$-energy forms on self-similar sets
in \cite{KS:GCDiff,KS:StrComp,KS:pqEnergySing}.%
\begin{thm}\label{thm:p-RF-harm-func-Hoelder}
Let $B\subset\ssset$ be non-empty and set
$B^{\domain_{p}}:=\bigcap_{u\in\domain_{p},\,u\vert_{B}=0}u^{-1}(0)$.
Let $x\in\ssset\setminus B^{\domain_{p}}$ and set 
$\resismet_{\form_{p}}(x,B):=\form_{p}\vert_{B\cup\{x\}}\bigl(\one^{B\cup\{x\}}_{x}\bigr)^{-1}$.
Then for any $y\in\ssset$,
\begin{equation}\label{eq:p-RF-Green-func-Hoelder}
\harext{\form_{p}}{B\cup\{x\}}\bigl[\one^{B\cup\{x\}}_{x}\bigr](y)
	\leq\frac{\resismet_{\form_{p}}(x,y)^{1/(p-1)}}{\resismet_{\form_{p}}(x,B)^{1/(p-1)}}.
\end{equation}
Moreover, for any $h\in\harfunc{\form_{p}}{B}$ with $\lVert h\rVert_{\sup,B}<\infty$
and any $y\in\ssset$,
\begin{equation}\label{eq:p-RF-harm-func-Hoelder}
\lvert h(x)-h(y)\rvert
	\leq\frac{\resismet_{\form_{p}}(x,y)^{1/(p-1)}}{\resismet_{\form_{p}}(x,B)^{1/(p-1)}}\osc_{B}[h].
\end{equation}
\end{thm}

The special case of \eqref{eq:p-RF-Green-func-Hoelder} with $B=\{z\}$ for $z\in\ssset\setminus\{x\}$,
the same inequality with $x$ and $z$ interchanged, and the equality
$\harext{\form_{p}}{\{x,z\}}\bigl[\one^{\{x,z\}}_{z}\bigr]
	=\one_{\ssset}-\harext{\form_{p}}{\{x,z\}}\bigl[\one^{\{x,z\}}_{x}\bigr]$
implied by \eqref{eq:p-RF-harext-homogeneous} together yield the triangle inequality
for $\resismet_{p}^{1/(p-1)}$ and thereby show the following corollary.
\begin{cor}\label{cor:p-RF-p-RM-triangle-ineq}
$\resismet_{p}^{1/(p-1)}\colon\ssset\times\ssset\to[0,\infty)$ is a metric on $\ssset$.
\end{cor}

We conclude this section by presenting an expression of $(\form_{p},\domain_{p})$ as the
``inductive limit'' of its traces $\{\form_{p}\vert_{V}\}_{V\subset\ssset,\,1\leq\#V<\infty}$
to finite subsets, which is a straightforward extension of the counterpart for resistance forms given
in \cite[Corollary 2.37]{K:Fractal2018} (see also \cite[Proof of Theorem 2.36-(2)]{K:Fractal2018}
and \cite[Theorems 2.3.6, 2.3.7 and Lemma 2.3.8]{Kig01}), and a few applications of it
to convergence in the seminorm $\form_{p}^{1/p}$.
\begin{thm}\label{thm:p-RF-inductive-limit}
It holds that
\begin{align}\label{eq:p-RF-inductive-limit-domain}
\domain_{p}&=\{u\in\mathbb{R}^{\ssset}\mid\sup\nolimits_{V\subset\ssset,\,1\leq\#V<\infty}\form_{p}\vert_{V}(u\vert_{V})<\infty\},\\
\form_{p}(u)&=\sup\nolimits_{V\subset\ssset,\,1\leq\#V<\infty}\form_{p}\vert_{V}(u\vert_{V})\quad\textrm{for any $u\in\domain_{p}$.}
\label{eq:p-RF-inductive-limit-form}
\end{align}
\end{thm}

Theorem \ref{thm:p-RF-inductive-limit} easily implies the following simple
characterization of the convergence in $\domain_{p}$ with respect to $\form_{p}^{1/p}$.
\begin{prop}\label{prop:p-RF-norm-convergence-characterize}
Let $u\in\domain_{p}$ and $\{u_{n}\}_{n\in\mathbb{N}}\subset\domain_{p}$.
\begin{enumerate}[label=\textup{(\arabic*)},align=left,leftmargin=*,topsep=4pt,itemsep=2pt]
\item\label{it:p-RF-upper-semicontinuous}Assume that $\lim_{n\to\infty}(u_{n}(x)-u_{n}(y))=u(x)-u(y)$
	for any $x,y\in\ssset$. Then $\form_{p}(u)\leq\liminf_{n\to\infty}\form_{p}(u_{n})$.
\item\label{it:p-RF-norm-convergence-characterize}$\lim_{n\to\infty}\form_{p}(u-u_{n})=0$
	if and only if $\limsup_{n\to\infty}\form_{p}(u_{n})\leq\form_{p}(u)$ and
	$\lim_{n\to\infty}(u_{n}(x)-u_{n}(y))=u(x)-u(y)$ for any $x,y\in\ssset$.
\end{enumerate}
\end{prop}

Proposition \ref{prop:p-RF-norm-convergence-characterize}-\ref{it:p-RF-norm-convergence-characterize}
further yields the following useful approximation results.%
\begin{prop}\label{prop:p-RF-approx}
\begin{enumerate}[label=\textup{(\arabic*)},align=left,leftmargin=*,topsep=4pt,itemsep=2pt]
\item\label{it:p-RF-approx-cutoff-given-func}Let
	$\{\varphi_{n}\}_{n\in\mathbb{N}}\subset\contfunc(\mathbb{R})$ satisfy
	$\lvert\varphi_{n}(t)-\varphi_{n}(s)\rvert\leq\lvert t-s\rvert$
	for any $n\in\mathbb{N}$ and any $s,t\in\mathbb{R}$ and
	$\lim_{n\to\infty}\varphi_{n}(t)=t$ for any $t\in\mathbb{R}$.
	Then $\{\varphi_{n}(u)\}_{n\in\mathbb{N}}\subset\domain_{p}$ and
	$\lim_{n\to\infty}\form_{p}(u-\varphi_{n}(u))=0$ for any $u\in\domain_{p}$.
\item\label{it:p-RF-approx-modify-given-seq}Let $u\in\domain_{p}$,
	$\{u_{n}\}_{n\in\mathbb{N}}\subset\domain_{p}$ and
	$\varphi\in\contfunc(\mathbb{R})$ satisfy
	$\lim_{n\to\infty}\form_{p}(u-u_{n})=0$,
	$\lim_{n\to\infty}u_{n}(x)=u(x)$ for some $x\in\ssset$,
	$\lvert\varphi(t)-\varphi(s)\rvert\leq\lvert t-s\rvert$
	for any $s,t\in\mathbb{R}$ and $\varphi(u)=u$.
	Then $\{\varphi(u_{n})\}_{n\in\mathbb{N}}\subset\domain_{p}$ and
	$\lim_{n\to\infty}\form_{p}(u-\varphi(u_{n}))=0$.
\end{enumerate}
\end{prop}
\begin{rmk}\label{rmk:p-RF-approx}
Typical choices of $\{\varphi_{n}\}_{n\in\mathbb{N}}\subset\contfunc(\mathbb{R})$
in Proposition \ref{prop:p-RF-approx}-\ref{it:p-RF-approx-cutoff-given-func} are
$\varphi_{n}(t)=(-n)\vee(t\wedge n)$ and $\varphi_{n}(t)=t-(-\frac{1}{n})\vee(t\wedge\frac{1}{n})$.
A typical use of Proposition \ref{prop:p-RF-approx}-\ref{it:p-RF-approx-modify-given-seq}
is to obtain a sequence of $I$-valued functions converging to $u$ in $\form_{p}^{1/p}$
when $I\subset\mathbb{R}$ is a closed interval and $u\in\domain_{p}$ is $I$-valued,
by considering $\varphi\in\contfunc(\mathbb{R})$ given by $\varphi(t):=(\inf I)\vee(t\wedge\sup I)$.
\end{rmk}
%
\section{Canonical $p$-resistance form on the Sierpi\'{n}ski gasket}\label{sec:SG-p-RF}
In this section, we introduce the (two-dimensional standard) Sierpi\'{n}ski gasket
and discuss the construction, the uniqueness and some basic properties of the
canonical $p$-resistance form on it.
\subsection{Sierpi\'{n}ski gasket}\label{ssec:SG}
To start with, the Sierpi\'{n}ski gasket is defined as follows.
\begin{dfn}[Sierpi\'{n}ski gasket]\label{dfn:SG}
Set $\ssindex:=\{1,2,3\}$, and
let $\ssvertices_{0}:=\{\SGvertex_{i}\mid i\in\ssindex\}\subset\mathbb{R}^{2}$
be the set of the vertices of an equilateral triangle
$\simplex\subset\mathbb{R}^{2}$, so that
$\simplex$ is the convex hull of $\ssvertices_{0}$ in $\mathbb{R}^{2}$.
We further define $\sseuc_{i}\colon\mathbb{R}^{2}\to\mathbb{R}^{2}$
by $\sseuc_{i}(x):=\SGvertex_{i}+\frac{1}{2}(x-\SGvertex_{i})$
for each $i\in\ssindex$, let $\ssset$ be the \emph{self-similar set} associated with
$\{\sseuc_{i}\}_{i\in\ssindex}$, i.e., the unique non-empty compact subset of $\mathbb{R}^{2}$
such that $\ssset=\bigcup_{i\in\ssindex}\sseuc_{i}(\ssset)$,
which exists and satisfies $\ssset\subsetneq\simplex$ thanks to
$\bigcup_{i\in\ssindex}\sseuc_{i}(\simplex)\subsetneq\simplex$
by \cite[Theorem 1.1.4]{Kig01}, and set $\ssmap_{i}:=\sseuc_{i}\vert_{\ssset}$
for each $i\in\ssindex$. The set $\ssset$ is called the
\emph{(two-dimensional standard) Sierpi\'{n}ski gasket}
(see Figure \ref{fig:SG} above). We equip $\ssset$ with the Euclidean metric
$\refmet\colon\ssset\times\ssset\to[0,\infty)$ given by $\refmet(x,y):=\lvert x-y\rvert$,
and the group of isometries of $(\ssset,\refmet)$ is denoted by $\sssym$.
\end{dfn}
%
\begin{figure}[t]\centering
	\includegraphics[height=140pt]{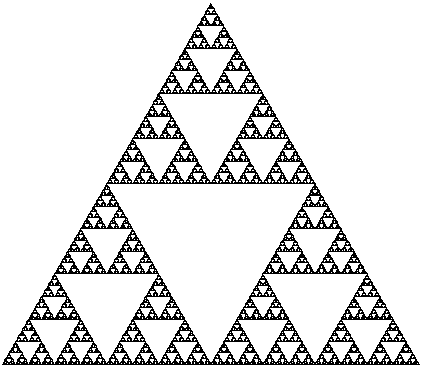}
	\caption{The (two-dimensional standard) Sierpi\'{n}ski gasket}
	\label{fig:SG}
\end{figure}

Throughout the rest of this paper, we fix the setting of Definition \ref{dfn:SG},
and use the following notation, which is standard in studying self-similar sets.
\begin{dfn}\label{dfn:words}
\begin{enumerate}[label=\textup{(\arabic*)},align=left,leftmargin=*,topsep=4pt,itemsep=2pt]
\item\label{it:words}We set $\words_{0}:=\{\emptyset\}$, where $\emptyset$ is an element called the \emph{empty word},
	$\words_{n}:=\ssindex^{n}=\{w_{1}\dots w_{n}\mid\textrm{$w_{i}\in\ssindex$ for $i\in\{1,\dots,n\}$}\}$
	for each $n\in\mathbb{N}$, and $\words_{*}:=\bigcup_{n\in\mathbb{N}\cup\{0\}}\words_{n}$.
	For $w\in\words_{*}$, the unique $n\in\mathbb{N}\cup\{0\}$ with $w\in\words_{n}$ is
	denoted by $\lvert w\rvert$ and called the \emph{length} of $w$.
\item\label{it:words-product}For $w=w_{1}\ldots w_{\lvert w\rvert},v=v_{1}\ldots v_{\lvert v\rvert}\in\words_{*}$,
	we define $wv\in\words_{*}$ by $wv:=w_{1}\ldots w_{\lvert w\rvert}v_{1}\ldots v_{\lvert v\rvert}$
	($w\emptyset:=w$, $\emptyset v:=v$). We also define
	$w^{(1)}\ldots w^{(k)}\in\words_{*}$ for $k\in\mathbb{N}\setminus\{1,2\}$
	and $w^{(1)},\ldots,w^{(k)}\in\words_{*}$
	inductively by $w^{(1)}\ldots w^{(k)}:=(w^{(1)}\ldots w^{(k-1)})w^{(k)}$.
	For $w\in\words_{*}$ and $n\in\mathbb{N}\cup\{0\}$ we set $w^{n}:=w\ldots w\in\words_{n\lvert w\rvert}$.
\item\label{it:words-ssmap}We set $\ssmap_{w}:=\ssmap_{w_{1}}\circ\cdots\circ\ssmap_{w_{n}}$
	($\ssmap_{\emptyset}:=\id_{\ssset}$) and $\ssset_{w}:=\ssmap_{w}(\ssset)$
	for each $w=w_{1}\ldots w_{n}\in\words_{*}$,
	$\ssvertices_{n}:=\bigcup_{w\in\words_{n}}\ssmap_{w}(\ssvertices_{0})$ for each $n\in\mathbb{N}$,
	and $\ssvertices_{*}:=\bigcup_{n\in\mathbb{N}\cup\{0\}}\ssvertices_{n}$,
	so that $\ssvertices_{n-1}\subsetneq\ssvertices_{n}$ for any $n\in\mathbb{N}$
	and $\ssvertices_{*}$ is dense in $\ssset$.
\end{enumerate}
\end{dfn}

The properties stated in the following proposition are crucial in developing the theory
of self-similar energy forms on the Sierpi\'{n}ski gasket $\ssset$.
\begin{prop}[Cf.\ {\cite[Proposition 1.3.5-(2) and Example 1.3.15]{Kig01}}]\label{prop:SG-intersecting-cells}
Let $w=w_{1}\ldots w_{\lvert w\rvert},v=v_{1}\ldots v_{\lvert v\rvert}\in\words_{*}\setminus\{\emptyset\}$
satisfy $w_{k}\not=v_{k}$ for some $k\in\{1,\ldots,\lvert w\rvert\wedge\lvert v\rvert\}$.
\begin{enumerate}[label=\textup{(\arabic*)},align=left,leftmargin=*,topsep=4pt,itemsep=2pt]
\item\label{it:SG-intersecting-cells}$\ssset_{w}\cap\ssset_{v}=\ssmap_{w}(\ssvertices_{0})\cap\ssmap_{v}(\ssvertices_{0})$.
\item\label{it:SG-intersecting-cells-words}$\ssset_{w}\cap\ssset_{v}\not=\emptyset$ if and only if there exist
	$\tau\in\words_{*}$, $n,m\in\mathbb{N}\cup\{0\}$ and $i,j\in\ssindex$ with $i\not=j$
	such that $w=\tau ij^{n}$ and $v=\tau ji^{m}$, in which case
	$\ssset_{w}\cap\ssset_{v}=\{\ssmap_{\tau}(\SGvertex_{ij})\}$,
	where $\SGvertex_{ij}:=\ssmap_{i}(\SGvertex_{j})=\ssmap_{j}(\SGvertex_{i})$.
\end{enumerate}
\end{prop}

The following proposition can be easily verified by showing, for any $g\in\sssym$,
first that $g(\ssvertices_{0})=\ssvertices_{0}$ and then by an induction on $n$
that $f^{-1}\circ g\vert_{\ssvertices_{n}}=\id_{\ssvertices_{n}}$ for any
$n\in\mathbb{N}$ for the isometry $f$ of $\mathbb{R}^{2}$ with
$f\vert_{\ssvertices_{0}}=g\vert_{\ssvertices_{0}}$.
\begin{prop}\label{prop:SG-symmetry}
$\sssym=\{f\vert_{\ssset}\mid\textrm{$f$ is an isometry of $\mathbb{R}^{2}$, $f(\ssvertices_{0})=\ssvertices_{0}$}\}$,
and the group $\sssym$ is generated by $\{\eucrefl{xy}\vert_{\ssset}\mid\textrm{$x,y\in\ssvertices_{0}$, $x\not=y$}\}$,
where $\eucrefl{xy}\colon\mathbb{R}^{2}\to\mathbb{R}^{2}$ denotes the reflection
in the line $\{z\in\mathbb{R}^{2}\mid\lvert z-x\rvert=\lvert z-y\rvert\}$.
\end{prop}
%
\subsection{Construction of the canonical $p$-resistance form}\label{ssec:SG-construct-p-RF}
Now we let $p\in(1,\infty)$ and fix it throughout the rest of this section.
The following lemma is immediate from Proposition \ref{prop:SG-intersecting-cells}-\ref{it:SG-intersecting-cells},
Lemma \ref{lem:p-RF1-p-RF5-cone} and the characterization of $p$-resistance forms on
finite sets given at the beginning of Example \ref{exmp:p-RF}-\ref{it:p-RF-finite}.%
\begin{lem}\label{lem:SG-p-RF-V0-Vn}
Let $\form_{p}^{(0)}$ be a $p$-resistance form on $\ssvertices_{0}$,
$\sscndc_{p}\in(0,\infty)$, $n\in\mathbb{N}\cup\{0\}$, and define
$\renom_{\sscndc_{p},n}(\form_{p}^{(0)})\colon\mathbb{R}^{\ssvertices_{n}}\to[0,\infty)$ by
\begin{equation}\label{eq:p-RF-V0-Vn}
\renom_{\sscndc_{p},n}(\form_{p}^{(0)})(u):=\sum_{w\in\words_{n}}\sscndc_{p}^{n}\form_{p}^{(0)}(u\circ\ssmap_{w}\vert_{\ssvertices_{0}}),
	\qquad u\in\mathbb{R}^{\ssvertices_{n}}.
\end{equation}
Then $\renom_{\sscndc_{p},n}(\form_{p}^{(0)})$ is a $p$-resistance form on $\ssvertices_{n}$.
\end{lem}

As presented in \cite[Sections 6 and 7]{Bar98}, \cite[Chapter 3]{Kig01} and
\cite[Chapter 1 and Section 4.2]{Str}, the problem of constructing a self-similar
Dirichlet form on a given p.-c.f.\ self-similar set is reduced to finding
$\sscndc_{2}\in(0,\infty)$ and a resistance form $\form_{2}^{(0)}$ on $\ssvertices_{0}$
satisfying $\renom_{\sscndc_{2},1}(\form_{2}^{(0)})\big\vert_{\ssvertices_{0}}=\form_{2}^{(0)}$
(or a generalization of it where $\sscndc_{2}=\sscndc_{2,i}$ is allowed to depend on
$i\in\ssindex$ as in \eqref{eq:ssform-intro}). The same is true also for
$p$-energy forms with general $p$, and such a construction was achieved first by
Herman, Peirone and Strichartz \cite{HPS}, for the Sierpi\'{n}ski gasket in
\cite[Corollary 3.7]{HPS} and for a class of p.-c.f.\ self-similar sets in
\cite[Theorem 5.8]{HPS}. Recently Cao, Gu and Qiu \cite{CGQ} have extended these
results in \cite{HPS} to much wider classes of p.-c.f.\ self-similar sets and
developed a theory of the construction of $p$-energy forms on general
p.-c.f.\ self-similar sets, extending many of the results in
\cite[Chapter 3]{Kig01} to general $p$. (See also \cite[Section 4.6]{Kig23},
where Kigami has constructed self-similar $p$-energy forms on a large class
of p.-c.f.\ self-similar sets, as a special case of his general result
\cite[Theorem 4.6]{Kig23} whose proof is based on a different method.)

A serious problem with all these results is that the constructed $p$-energy forms
are explicitly claimed to satisfy only \eqref{eq:form-unit-contraction-intro},
which is far from strong enough for further detailed analysis of them.
Our contribution in this context is that we have identified \ref{it:p-RF5}
as the right property to assume for $p$-energy forms and have verified that
the constructions in \cite{HPS,CGQ,Kig23} can be modified or seen to yield
$p$-energy forms satisfying \ref{it:p-RF5}. The details of this result will
appear in \cite{KS:GCDiff}. In the present setting of the Sierpi\'{n}ski gasket,
a version of it can be stated as follows. Recall Theorem \ref{thm:p-RF-trace}
and Definition \ref{dfn:p-RF-trace} for traces of $p$-resistance forms.%
\begin{thm}\label{thm:SG-p-RF-V0-exists}
There exists a unique $\sscndc_{p}\in(0,\infty)$ such that
$\renom_{\sscndc_{p},1}(\form_{p}^{(0)})\big\vert_{\ssvertices_{0}}=\form_{p}^{(0)}$
for some $p$-resistance form $\form_{p}^{(0)}$ on $\ssvertices_{0}$.
Moreover, $\form_{p}^{(0)}$ can be chosen so as to be \emph{$\sssym$-invariant},
i.e., satisfy $\form_{p}^{(0)}(u\circ g\vert_{\ssvertices_{0}})=\form_{p}^{(0)}(u)$
for any $u\in\mathbb{R}^{\ssvertices_{0}}$ and any $g\in\sssym$.
\end{thm}
\begin{prfsktch}
The uniqueness of $\sscndc_{p}$ is immediate from $\#\ssvertices_{0}<\infty$,
\ref{it:p-RF1} for $\form_{p}^{(0)}$ and the fact that 
$\renom_{\sscndc_{p},n}(\form_{p}^{(0)})\big\vert_{\ssvertices_{n}}=\form_{p}^{(0)}$
for any $n\in\mathbb{N}$ by Proposition \ref{prop:SG-intersecting-cells}-\ref{it:SG-intersecting-cells},
$\renom_{\sscndc_{p},1}(\form_{p}^{(0)})\big\vert_{\ssvertices_{0}}=\form_{p}^{(0)}$
and Proposition \ref{prop:p-RF-trace-compatible}.
For existence, by \cite[Theorems 6.3 and 4.2]{CGQ} there exist
$\sscndc_{p}\in(0,\infty)$ and $\form_{p}^{(0)}\in\mathcal{Q}_{p}(\ssvertices_{0})$ with
the desired properties except \ref{it:p-RF5}, where $\mathcal{Q}_{p}(\ssvertices_{0})$
is as defined in \cite[Definition 2.8]{CGQ}, but we see from the
definition of $\mathcal{Q}_{p}(\ssvertices_{0})$, Example \ref{exmp:p-RF}-\ref{it:p-RF-finite}
and Theorem \ref{thm:p-RF-trace} that any $E\in\mathcal{Q}_{p}(\ssvertices_{0})$ in fact
satisfies \ref{it:p-RF5} and in particular $\form_{p}^{(0)}$ does, completing the proof.
\qed\end{prfsktch}

Throughout the rest of this paper, we let $\sscndc_{p}$ denote the unique element of
$(0,\infty)$ as in Theorem \ref{thm:SG-p-RF-V0-exists} and fix a $\sssym$-invariant
$p$-resistance form $\form_{p}^{(0)}$ on $\ssvertices_{0}$ satisfying
$\renom_{\sscndc_{p},1}(\form_{p}^{(0)})\big\vert_{\ssvertices_{0}}=\form_{p}^{(0)}$.
As stated in Theorem \ref{thm:SG-p-RF-V0-unique} below, $\form_{p}^{(0)}$
will turn out to be unique up to constant multiples.

The following proposition is an easy consequence of
Proposition \ref{prop:SG-intersecting-cells}-\ref{it:SG-intersecting-cells},
$\renom_{\sscndc_{p},1}(\form_{p}^{(0)})\big\vert_{\ssvertices_{0}}=\form_{p}^{(0)}$
and Proposition \ref{prop:p-RF-trace-compatible}.
\begin{prop}\label{prop:SG-p-RF-compatible-seq}
$\renom_{\sscndc_{p},n+m}(\form_{p}^{(0)})\big\vert_{\ssvertices_{n}}
	=\renom_{\sscndc_{p},n}(\form_{p}^{(0)})$
for any $n,m\in\mathbb{N}\cup\{0\}$.
\end{prop}

As a reflection of the fact that the Sierpi\'{n}ski gasket is a p.-c.f.\ self-similar
set, we also have the following important feature of $\sscndc_{p}$.
\begin{prop}[{\cite[Lemma 3.8 and Theorem 5.9]{HPS}}, {\cite[Lemma 5.4]{CGQ}}]\label{prop:SG-p-RF-resis-scale-less-than-1}
$\sscndc_{p}>1$.
\end{prop}

Based on Propositions \ref{prop:SG-p-RF-compatible-seq} and \ref{prop:SG-p-RF-resis-scale-less-than-1},
the standard machinery for constructing the ``inductive limit'' of $p$-energy forms
as presented in \cite[Proposition 5.3]{CGQ}, which is an adaptation of the relevant
pieces of the theory of resistance forms due to \cite[Sections 2.2, 2.3 and 3.3]{Kig01},
gives the following result.
\begin{dfn}\label{dfn:SG-p-RF}
We set $\form_{p}^{(n)}:=\renom_{\sscndc_{p},n}(\form_{p}^{(0)})$ for each $n\in\mathbb{N}$,
and define a linear subspace $\domain_{p,*}$ of $\mathbb{R}^{\ssvertices_{*}}$, one
$\domain_{p}$ of $\contfunc(\ssset)$ and a functional $\form_{p}\colon\domain_{p}\to[0,\infty)$ by
\begin{align}\label{eq:SG-p-RF-domain}
\domain_{p}&:=\Bigl\{u\in\contfunc(\ssset)\Bigm\vert\lim_{n\to\infty}\form_{p}^{(n)}(u\vert_{\ssvertices_{n}})<\infty\Bigr\},\\
\form_{p}(u)&:=\lim_{n\to\infty}\form_{p}^{(n)}(u\vert_{\ssvertices_{n}})\in[0,\infty),
	\quad u\in\domain_{p};
\label{eq:SG-p-RF-form}
\end{align}
note that $\{\form_{p}^{(n)}(u\vert_{\ssvertices_{n}})\}_{n\in\mathbb{N}\cup\{0\}}\subset[0,\infty)$
is non-decreasing by Proposition \ref{prop:SG-p-RF-compatible-seq} and hence has a limit in $[0,\infty]$
for any $u\in\mathbb{R}^{\ssvertices_{*}}$.
\end{dfn}
\begin{thm}[Cf.\ {\cite[Proposition 5.3]{CGQ}}]\label{thm:SG-p-RF}
$(\form_{p},\domain_{p})$ is a $p$-resistance form on $\ssset$ with
$\form_{p}\vert_{\ssvertices_{n}}=\form_{p}^{(n)}$ for any $n\in\mathbb{N}\cup\{0\}$
and with the metric $\resismet_{\form_{p}}^{1/(p-1)}$
\textup{(recall Corollary \ref{cor:p-RF-p-RM-triangle-ineq})} compatible with the original
(Euclidean) topology of $\ssset$, and satisfies the following \emph{self-similarity}:
\begin{enumerate}[label=\textup{(SSE\arabic*)},align=left,leftmargin=*,topsep=4pt,itemsep=2pt]
\item\label{it:SG-p-RF-SSE1}$\domain_{p}=\{u\in\contfunc(\ssset)\mid\textrm{$u\circ\ssmap_{i}\in\domain_{p}$ for any $i\in\ssindex$}\}$.
\item\label{it:SG-p-RF-SSE2}$\form_{p}(u)=\sum_{i\in\ssindex}\sscndc_{p}\form_{p}(u\circ\ssmap_{i})$ for any $u\in\domain_{p}$.
\end{enumerate}
Moreover, $(\form_{p},\domain_{p})$ is \emph{$\sssym$-invariant}, i.e., $u\circ g\in\domain_{p}$
and $\form_{p}(u\circ g)=\form_{p}(u)$ for any $u\in\domain_{p}$ and any $g\in\sssym$.
\end{thm}

Below are some basic properties of $(\form_{p},\domain_{p})$. First, the following simple
lemma underlies the monotonicity of $\sscndc_{p}^{1/(p-1)}$ as a function of $p$
(see Theorem \ref{thm:SG-p-RF-resis-scale-decreasing} below), which is a key observation
in our study of $p$-energy measures in \cite{KS:pqEnergySing}.
\begin{lem}\label{lem:SG-p-RF-RM-contraction}
For any $w\in\words_{*}$ and any $x,y\in\ssset$,
\begin{equation}\label{eq:SG-p-RF-RM-contraction}
\resismet_{\form_{p}}(\ssmap_{w}(x),\ssmap_{w}(y))
	\leq\sscndc_{p}^{-\lvert w\rvert}\resismet_{\form_{p}}(x,y).
\end{equation}
\end{lem}
\begin{proof}
This is immediate from \ref{it:SG-p-RF-SSE1}, \ref{it:SG-p-RF-SSE2} and \ref{it:p-RF4}.
\qed\end{proof}

From the above construction of $(\form_{p},\domain_{p})$, we easily obtain the following
characterizations of $\form_{p}$-harmonic functions on $\ssset\setminus\ssvertices_{n}$
for $n\in\mathbb{N}\cup\{0\}$. Recall Definition \ref{dfn:p-RF-harm-func} and Theorem \ref{thm:p-RF-trace}.
\begin{prop}\label{prop:SG-p-RF-harmonic-Vn}
Let $n\in\mathbb{N}\cup\{0\}$. Then for each $h\in\contfunc(\ssset)$,
the following three conditions are equivalent to each other:
\begin{enumerate}[label=\textup{(\arabic*)},align=left,leftmargin=*,topsep=4pt,itemsep=2pt]
\item\label{it:SG-p-RF-harmonic-Vn}$h\in\harfunc{\form_{p}}{\ssvertices_{n}}$.
\item\label{it:SG-p-RF-harmonic-Vn-self-sim}$h\circ\ssmap_{w}\in\harfunc{\form_{p}}{\ssvertices_{0}}$
	for any $w\in\words_{n}$.
\item\label{it:SG-p-RF-harmonic-Vn-test}For any $m\in\mathbb{N}$ with $m>n$ and any $x\in\ssvertices_{m}\setminus\ssvertices_{n}$,
	\begin{equation}\label{eq:SG-p-RF-harmonic-Vn-test}
	\sum_{w\in\words_{m},\,x\in\ssmap_{w}(\ssvertices_{0})}\sscndc_{p}^{m}\form_{p}^{(0)}\bigl(h\circ\ssmap_{w}\vert_{\ssvertices_{0}};\one_{\ssmap_{w}^{-1}(x)}\bigr)=0.
	\end{equation}
\end{enumerate}
\end{prop}

The implication from \ref{it:SG-p-RF-harmonic-Vn} to \ref{it:SG-p-RF-harmonic-Vn-self-sim}
in Proposition \ref{prop:SG-p-RF-harmonic-Vn} enables us to conclude the following
localized version of the weak comparison principle (Proposition \ref{prop:p-RF-comp}).
\begin{prop}[A localized weak comparison principle]\label{prop:SG-p-RF-harmonic-Vn-comp}
Let $n\in\mathbb{N}\cup\{0\}$, $w\in\words_{n}$, and let
$u,v\in\harfunc{\form_{p}}{\ssvertices_{n}}$ satisfy $u(x)\leq v(x)$ for any
$x\in\ssmap_{w}(\ssvertices_{0})$. Then $u(x)\leq v(x)$ for any $x\in\ssset_{w}$.
\end{prop}

The following proposition is useful in reducing the proof of a statement for general $u\in\domain_{p}$
to the case of $u\in\bigcup_{n\in\mathbb{N}\cup\{0\}}\harfunc{\form_{p}}{\ssvertices_{n}}$.
\begin{prop}\label{prop:SG-p-RF-harmonic-approx}
Let $u\in\domain_{p}$ and set $u_{n}:=\harext{\form_{p}}{\ssvertices_{n}}[u\vert_{\ssvertices_{n}}]$
for each $n\in\mathbb{N}\cup\{0\}$. Then $\lim_{n\to\infty}\form_{p}(u-u_{n},u-u_{n})=0$.
\end{prop}
\begin{proof}
This is immediate from \eqref{eq:p-RF-sClarkson-pleq2} and \eqref{eq:p-RF-wClarkson-pgeq2}
with $\frac{u+u_{n}}{2},\frac{u-u_{n}}{2}$ in place of $u,v$,
$\form_{p}\bigl(\frac{u+u_{n}}{2}\bigr)\geq\form_{p}(u_{n})=\form_{p}^{(n)}(u\vert_{\ssvertices_{n}})$
for $n\in\mathbb{N}\cup\{0\}$ and \eqref{eq:SG-p-RF-form}.
\qed\end{proof}

We collect some important consequences of Theorem \ref{thm:SG-p-RF}
in relation to the topology of $\ssset$ in the following theorem.
\begin{thm}[Cf.\ {\cite[Theorem 5.2]{CGQ}}]\label{thm:SG-p-RF-regular-str-local}
\begin{enumerate}[label=\textup{(\arabic*)},align=left,leftmargin=*,topsep=4pt,itemsep=2pt]
\item\label{it:SG-p-RF-regular}$(\form_{p},\domain_{p})$ is \emph{regular}, i.e.,
	$\domain_{p}$ is a dense subalgebra of $(\contfunc(\ssset),\lVert\cdot\rVert_{\sup})$.
\item\label{it:SG-p-RF-strongly-local}$(\form_{p},\domain_{p})$ is
	\emph{strongly local}, i.e., $\form_{p}(u_{1};v)=\form_{p}(u_{2};v)$ for any
	$u_{1},u_{2},v\in\domain_{p}$ that satisfy $(u_{1}(x)-u_{2}(x)-a)(v(x)-b)=0$
	for any $x\in\ssset$ for some $a,b\in\mathbb{R}$.
\end{enumerate}
\end{thm}
\begin{proof}
\begin{enumerate}[label=\textup{(\arabic*)},align=left,leftmargin=*,topsep=4pt,itemsep=2pt]
\item\label{it:SG-p-RF-regular-prf}This follows from the compactness of $\ssset$,
	\ref{it:p-RF1}, \ref{it:p-RF3}, Proposition \ref{prop:p-RF5-conseq}-\ref{it:p-RF-product}
	and the Stone--Weierstrass theorem (see, e.g., \cite[Theorem 2.4.11]{Dud}).
\item\label{it:SG-p-RF-strongly-local-prf}Under the stronger assumption
	$\supp_{\ssset}[u_{1}-u_{2}-a\one_{\ssset}]\cap\supp_{\ssset}[v-b\one_{\ssset}]=\emptyset$
	this is immediate from the compactness of $\ssset$, \ref{it:SG-p-RF-SSE1}, \ref{it:SG-p-RF-SSE2}
	and \eqref{eq:p-RF-Diff-dfn}, and the general case follows by applying this special case
	with $v$ replaced by $\varphi_{n}(v-b\one_{\ssset})$ for $\varphi_{n}\in\contfunc(\mathbb{R})$
	given by $\varphi_{n}(t)=t-(-\frac{1}{n})\vee(t\wedge\frac{1}{n})$ and using
	Proposition \ref{prop:p-RF-approx}-\ref{it:p-RF-approx-cutoff-given-func}
	to let $n\to\infty$.
\qed\end{enumerate}
\end{proof}

The following proposition is used to prove the weak comparison principle for arbitrary
open subsets of $\ssset$ (Proposition \ref{prop:SG-p-RF-comp}), which is an intermediate
step of the proof of the strong one (Theorem \ref{thm:SG-p-RF-strong-comp} below).
Detailed proofs of Propositions \ref{prop:SG-p-RF-local-redefinition} and \ref{prop:SG-p-RF-comp}
will be given in \cite{KS:GCDiff} in the setting of a local $p$-resistance form.
\begin{prop}\label{prop:SG-p-RF-local-redefinition}
Let $U$ be an open subset of $\ssset$ and let $u\in\domain_{p}$ satisfy $u(x)=0$
for any $x\in\partial_{\ssset}U$. Then $u\one_{U}\in\domain_{p}$.
\end{prop}
\begin{prfsktch}
Let $n\in\mathbb{N}$, define $\varphi_{n}\in\contfunc(\mathbb{R})$ by
$\varphi_{n}(t)=t-(-\frac{1}{n})\vee(t\wedge\frac{1}{n})$ and set
$A_{n}:=U\cap\supp_{\ssset}[\varphi_{n}(u)]$. Since
$A_{n}=\overline{U}^{\ssset}\cap\supp_{\ssset}[\varphi_{n}(u)]$ by $u\vert_{\partial_{\ssset}U}=0$,
$A_{n}$ is a compact subset of $U$, hence by
Theorem \ref{thm:SG-p-RF-regular-str-local}-\ref{it:SG-p-RF-regular}
and Proposition \ref{prop:p-RF5-conseq}-\ref{it:p-RF-1Lip}
we can take $v_{n}\in\domain_{p}$ such that
$\one_{A_{n}}(x)\leq v_{n}(x)\leq\one_{U}(x)$ for any $x\in\ssset$,
and then $\varphi_{n}(u)\one_{U}=\varphi_{n}(u)v_{n}\in\domain_{p}$
by Proposition \ref{prop:p-RF5-conseq}-\ref{it:p-RF-product}. Now it is not difficult
to see that $\{\varphi_{n}(u)\one_{U}\}_{n\in\mathbb{N}}$ gives a Cauchy sequence in
$(\domain_{p}/\mathbb{R}\one_{\ssset},\form_{p}^{1/p})$
by Theorem \ref{thm:SG-p-RF-regular-str-local}-\ref{it:SG-p-RF-strongly-local}
and Proposition \ref{prop:p-RF-approx}-\ref{it:p-RF-approx-cutoff-given-func} 
and thus converges in norm in $(\domain_{p}/\mathbb{R}\one_{\ssset},\form_{p}^{1/p})$
to its pointwise limit $u\one_{U}$ by \ref{it:p-RF2} and Proposition \ref{prop:p-RF4-conseq},
whence $u\one_{U}\in\domain_{p}$.
\qed\end{prfsktch}
\begin{prop}[Weak comparison principle]\label{prop:SG-p-RF-comp}
Let $U\subsetneq\ssset$ be an open subset of $\ssset$ and let
$u,v\in\harfunc{\form_{p}}{\ssset\setminus U}$ satisfy $u(x)\leq v(x)$
for any $x\in\partial_{\ssset}U$. Then $u(x)\leq v(x)$ for any $x\in U$.
\end{prop}
\begin{prfsktch}
Setting $f:=u-(u-v)^{+}\one_{\ssset\setminus U}$ and $g:=v+(u-v)^{+}\one_{\ssset\setminus U}$,
we have $f,g\in\domain_{p}$ by Proposition \ref{prop:p-RF5-conseq}-\ref{it:p-RF-1Lip},
$u\vert_{\partial_{\ssset}U}\leq v\vert_{\partial_{\ssset}U}$ and
Proposition \ref{prop:SG-p-RF-local-redefinition},
$f,g\in\harfunc{\form_{p}}{\ssset\setminus U}$ by
Theorem \ref{thm:SG-p-RF-regular-str-local}-\ref{it:SG-p-RF-strongly-local}
and $u,v\in\harfunc{\form_{p}}{\ssset\setminus U}$,
$f(x)=(u\wedge v)(x)\leq(u\vee v)(x)=g(x)$ for any $x\in\ssset\setminus U$, and thus
$u(x)=f(x)\leq g(x)=v(x)$ for any $x\in U$ by Proposition \ref{prop:p-RF-comp}.
\qed\end{prfsktch}
%
\subsection{Strong comparison principle and the uniqueness of the form}\label{ssec:SG-strong-comp-p-RF-unique}
While the weak comparison principle as in Propositions \ref{prop:SG-p-RF-harmonic-Vn-comp}
and \ref{prop:SG-p-RF-comp} above is already good enough in many applications, for detailed
analysis of $\form_{p}$-harmonic functions it is very important to have the strong
comparison principle as stated in Theorem \ref{thm:SG-p-RF-strong-comp} below.
In fact, it can be proved for general p.-c.f.\ self-similar sets by
combining Proposition \ref{prop:p-RF-Diff-monotone} with the self-similarity of
$(\form_{p},\domain_{p})$ and $\#\ssvertices_{0}<\infty$, and allows us to apply
a nonlinear version of Perron--Frobenius theory presented in \cite[Chapters 5 and 6]{LN}
and thereby to conclude the uniqueness of $\form_{p}^{(0)}$ under the assumption of some good
geometric symmetry of the self-similar set by following \cite[Proof of Theorem 2.6]{Pei}.
The details of these results will be provided in \cite{KS:StrComp}. In this subsection,
we give their precise statements and a brief sketch of the proofs of them
in the current setting of the Sierpi\'{n}ski gasket. First, the strong comparison
principle can be stated as follows.
\begin{thm}[Strong comparison principle]\label{thm:SG-p-RF-strong-comp}
Let $U\subsetneq\ssset$ be a connected open subset of $\ssset$ and let
$u,v\in\harfunc{\form_{p}}{\ssset\setminus U}$ satisfy $u(x)\leq v(x)$
for any $x\in\partial_{\ssset}U$. Then either $u(x)<v(x)$ for any $x\in U$
or $u(x)=v(x)$ for any $x\in\overline{U}^{\ssset}$.
\end{thm}
\begin{prfsktch}
For $x\in\ssset$ and $m\in\mathbb{N}$ we set
$\ssset_{m,x}:=\bigcup_{w\in\words_{m},\,x\in\ssset_{w}}\ssset_{w}$ and
$\ssnbd_{m,x}:=\bigcup_{y\in\ssset_{m,x}\cap\ssvertices_{m}}\ssset_{m,y}$.
First, we see from \ref{it:SG-p-RF-SSE1}, \ref{it:SG-p-RF-SSE2}, \eqref{eq:p-RF-Diff-dfn},
\eqref{eq:p-RF-trace-Diff} and $\form_{p}\vert_{\ssvertices_{0}}=\form_{p}^{(0)}$ from
Theorem \ref{thm:SG-p-RF} that for any $h\in\harfunc{\form_{p}}{\ssset\setminus U}$,
\begin{gather}\label{eq:SG-p-RF-harmonic-self-sim-general}
h\circ\ssmap_{w}\in\harfunc{\form_{p}}{\ssvertices_{0}}
	\quad\textrm{for any $w\in\words_{*}$ with $\ssset_{w}\subset U$, and}\\
\textrm{\eqref{eq:SG-p-RF-harmonic-Vn-test} holds for any $m\in\mathbb{N}$
and any $x\in\ssvertices_{m}$ with $\ssset_{m,x}\subset U$.}
\label{eq:SG-p-RF-harmonic-test-general}
\end{gather}

Let $q\in U\cap\ssvertices_{*}$, let $m\in\mathbb{N}$ satisfy $q\in\ssvertices_{m}$
and $\ssnbd_{m,q}\subset U$, let $w\in\words_{m}$ satisfy $q\in\ssset_{w}$,
and assume that $u(q)=v(q)$. We claim that
$u\circ\ssmap_{w}=v\circ\ssmap_{w}$. Suppose to the contrary that
$B:=\{y\in\ssvertices_{0}\mid u\circ\ssmap_{w}(y)<v\circ\ssmap_{w}(y)\}\not=\emptyset$.
Then setting $a:=\min_{y\in B}(v\circ\ssmap_{w}(y)-u\circ\ssmap_{w}(y))$ and
$f:=u\circ\ssmap_{w}\vert_{\ssvertices_{0}}+a\one_{B}$, for any $y\in\ssvertices_{0}\setminus B$
we would be able to see from \eqref{eq:SG-p-RF-harmonic-test-general} with
$x=\ssmap_{w}(y)$ and Proposition \ref{prop:p-RF-Diff-monotone} that
\begin{equation}\label{eq:SG-p-RF-strong-comp-derivative-stay}
\form_{p}^{(0)}(u\circ\ssmap_{w}\vert_{\ssvertices_{0}};\one_{y})
	=\form_{p}^{(0)}(f;\one_{y})
	=\form_{p}^{(0)}(v\circ\ssmap_{w}\vert_{\ssvertices_{0}};\one_{y}),
\end{equation}
whereas by $\ssmap_{w}^{-1}(q)\in\ssvertices_{0}\setminus B\not=\emptyset\not=B$,
\eqref{eq:p-RF-Diff-strictly-convex} and $a>0$ we would also have
\begin{equation}\label{eq:SG-p-RF-strong-comp-derivative-incr}
\form_{p}^{(0)}(u\circ\ssmap_{w}\vert_{\ssvertices_{0}};\one_{B})<\form_{p}^{(0)}(f;\one_{B}).
\end{equation}
The conjunction of \eqref{eq:SG-p-RF-strong-comp-derivative-stay} and
\eqref{eq:SG-p-RF-strong-comp-derivative-incr} would contradict
$\form_{p}^{(0)}(u\circ\ssmap_{w}\vert_{\ssvertices_{0}};\one_{\ssvertices_{0}})
	=0=\form_{p}^{(0)}(f;\one_{\ssvertices_{0}})$,
proving $u\circ\ssmap_{w}\vert_{\ssvertices_{0}}=v\circ\ssmap_{w}\vert_{\ssvertices_{0}}$
and hence $u\circ\ssmap_{w}=v\circ\ssmap_{w}$ by \eqref{eq:SG-p-RF-harmonic-self-sim-general}.

Finally, if $x\in U\setminus\ssvertices_{*}$ and $u(x)=v(x)$, then choosing
$m\in\mathbb{N}$ and $w\in\words_{m}$ so that $x\in\ssset_{w}$ and
$\bigcup_{q\in\ssmap_{w}(\ssvertices_{0})}\ssnbd_{m,q}\subset U$,
we have $u(q)=v(q)$ for some $q\in\ssmap_{w}(\ssvertices_{0})$ by 
\eqref{eq:SG-p-RF-harmonic-self-sim-general}, \eqref{eq:p-RF-harext-homogeneous},
Proposition \ref{prop:p-RF-comp} and $u(x)=v(x)$, and therefore
$u\circ\ssmap_{w}=v\circ\ssmap_{w}$ by the previous paragraph.
It follows that $\{x\in U\mid u(x)=v(x)\}$ is both open and closed in $U$
and thus either $\emptyset$ or $U$ by the connectedness of $U$.
\qed\end{prfsktch}

Theorem \ref{thm:SG-p-RF-strong-comp} has an important application to identifying the
principal term of the asymptotic behavior of $h\in\harfunc{\form_{p}}{\ssvertices_{0}}$
at each point of $\ssvertices_{0}=\{\SGvertex_{i}\mid i\in\ssindex\}$ as proved in
Theorem \ref{thm:SG-p-RF-Perron-Frobenius} below. For this purpose we need the
following proposition, which indicates what the asymptotic decay rate here should be.
In fact, part of Proposition \ref{prop:SG-p-RF-normal-derivative-self-sim}-\ref{it:SG-p-RF-normal-derivative-self-sim},\ref{it:SG-p-RF-normal-derivative-monotone},\ref{it:SG-p-RF-harmonic-eigenfunc}
and of Theorem \ref{thm:SG-p-RF-Perron-Frobenius} was obtained in \cite[Section 5]{SW}
under the suppositions of \eqref{eq:p-RF-Diff-dfn}, the strict convexity of
$(\domain_{p}/\mathbb{R}\one_{\ssset},\form_{p}^{1/p})$ implied by
\eqref{eq:p-RF-sClarkson-pleq2} and \eqref{eq:p-RF-wClarkson-pgeq2}, and some
special cases of Theorem \ref{thm:SG-p-RF-strong-comp} and Proposition \ref{prop:p-RF-comp},
but we provide here self-contained proofs of them for the reader's convenience.
\begin{prop}\label{prop:SG-p-RF-normal-derivative-self-sim}
Let $i\in\ssindex$, $j\in\ssindex\cap\{i-2,i+1\}$ and $k\in\ssindex\cap\{i-1,i+2\}$.
\begin{enumerate}[label=\textup{(\arabic*)},align=left,leftmargin=*,topsep=4pt,itemsep=2pt]
\item\label{it:SG-p-RF-normal-derivative-self-sim}$\form_{p}^{(0)}(h\vert_{\ssvertices_{0}};\one_{\SGvertex_{i}})
		=\sscndc_{p}^{n}\form_{p}^{(0)}(h\circ\ssmap_{i^{n}}\vert_{\ssvertices_{0}};\one_{\SGvertex_{i}})$
	for any $h\in\harfunc{\form_{p}}{\ssvertices_{0}}$ and any $n\in\mathbb{N}$.
\item\label{it:SG-p-RF-normal-derivative-monotone}$\mathbb{R}\ni t\mapsto\form_{p}^{(0)}(a\one_{\SGvertex_{j}}+t\one_{\SGvertex_{k}};\one_{\SGvertex_{i}})\in\mathbb{R}$
	and $\mathbb{R}\ni t\mapsto\form_{p}^{(0)}(t\one_{\SGvertex_{j}}-a\one_{\SGvertex_{k}};\one_{\SGvertex_{i}})\in\mathbb{R}$
	are strictly decreasing and $\form_{p}^{(0)}(a\one_{\SGvertex_{j}}-a\one_{\SGvertex_{k}};\one_{\SGvertex_{i}})=0$
	for any $a\in\mathbb{R}$.
\item\label{it:SG-p-RF-harmonic-eigenfunc}Define
	$\harprincipal{p}{i}:=\harext{\form_{p}}{\ssvertices_{0}}[\one_{\SGvertex_{j}}+\one_{\SGvertex_{k}}]$
	and $\harsecond{p}{i}:=\harext{\form_{p}}{\ssvertices_{0}}[\one_{\SGvertex_{k}}-\one_{\SGvertex_{j}}]$.
	Then $\harprincipal{p}{i}\circ\ssmap_{i}=\sscndc_{p}^{-1/(p-1)}\harprincipal{p}{i}$, and there exists
	$\lambda_{p}\in(0,\sscndc_{p}^{-1/(p-1)})$ such that $\harsecond{p}{i}\circ\ssmap_{i}=\lambda_{p}\harsecond{p}{i}$.
\item\label{it:SG-p-RF-harmonic-osc-contraction}Let $h\in\harfunc{\form_{p}}{\ssvertices_{0}}$.
	Then $\osc_{\ssset}[h\circ\ssmap_{i}]=\sscndc_{p}^{-1/(p-1)}\osc_{\ssset}[h]$ if $h(\SGvertex_{j})=h(\SGvertex_{k})$,
	and $0<\osc_{\ssset}[h\circ\ssmap_{i}]<\sscndc_{p}^{-1/(p-1)}\osc_{\ssset}[h]$ if $h(\SGvertex_{j})\not=h(\SGvertex_{k})$.
\item\label{it:SG-p-RF-harmonic-non-degenerate}$h\circ\ssmap_{i}\in\harfunc{\form_{p}}{\ssvertices_{0}}\setminus\mathbb{R}\one_{\ssset}$
	for any $h\in\harfunc{\form_{p}}{\ssvertices_{0}}\setminus\mathbb{R}\one_{\ssset}$.
\end{enumerate}
\end{prop}
\begin{proof}
\begin{enumerate}[label=\textup{(\arabic*)},align=left,leftmargin=*,topsep=4pt,itemsep=2pt]
\item\label{it:SG-p-RF-normal-derivative-self-sim-prf}This is the counterpart of
	\cite[Corollary 3.2.2]{Kig01} and can be proved in the same way. Indeed,
	since $h\vert_{\ssvertices_{n}}\in\harfunc{\form_{p}^{(n)}}{\ssvertices_{0}}$
	by $\form_{p}\vert_{\ssvertices_{n}}=\form_{p}^{(n)}$ from Theorem \ref{thm:SG-p-RF}
	and Proposition \ref{prop:p-RF-trace-compatible},
	$\form_{p}^{(0)}(h\vert_{\ssvertices_{0}};\one_{\SGvertex_{i}})
		=\form_{p}^{(n)}(h\vert_{\ssvertices_{n}};\one^{\ssvertices_{n}}_{\SGvertex_{i}})$
	by $\form_{p}^{(0)}=\form_{p}^{(n)}\vert_{\ssvertices_{0}}$ from
	Proposition \ref{prop:SG-p-RF-compatible-seq} and \eqref{eq:p-RF-trace-Diff}, and
	the assertion follows from this last equality, the definition \eqref{eq:p-RF-V0-Vn} of
	$\form_{p}^{(n)}=\renom_{\sscndc_{p},n}(\form_{p}^{(0)})$ and \eqref{eq:p-RF-Diff-dfn}.
\item\label{it:SG-p-RF-normal-derivative-monotone-prf}For the first claim, since
	$\harext{\form_{p}}{\ssvertices_{0}}[t\one_{\SGvertex_{j}}+s\one_{\SGvertex_{k}}](\ssmap_{i}(\SGvertex_{j}))$
	and $\harext{\form_{p}}{\ssvertices_{0}}[t\one_{\SGvertex_{j}}+s\one_{\SGvertex_{k}}](\ssmap_{i}(\SGvertex_{k}))$
	are strictly increasing in each of $t,s\in\mathbb{R}$ by Theorem \ref{thm:SG-p-RF-strong-comp},
	in view of \ref{it:SG-p-RF-normal-derivative-self-sim-prf} with $n=1$
	it suffices to show that
	$\form_{p}^{(0)}(u+t\one_{\SGvertex_{j}}+s\one_{\SGvertex_{k}};\one_{\SGvertex_{i}})
		<\form_{p}^{(0)}(u;\one_{\SGvertex_{i}})$
	for any $u\in\mathbb{R}^{\ssvertices_{0}}$ and any $t,s\in(0,\infty)$, which in
	turn follows from Proposition \ref{prop:p-RF-Diff-monotone} and the fact that
	$\form_{p}^{(0)}(u+t\one_{\SGvertex_{j}}+t\one_{\SGvertex_{k}};\one_{\SGvertex_{i}})
		=\form_{p}^{(0)}(u-t\one_{\SGvertex_{i}};\one_{\SGvertex_{i}})
		<\form_{p}^{(0)}(u;\one_{\SGvertex_{i}})$
	for any $t\in(0,\infty)$ by \eqref{eq:p-RF-Diff-homogeneous} and
	\eqref{eq:p-RF-Diff-strictly-convex}. The latter assertion is immediate from
	the invariance of $\form_{p}^{(0)}$ under $\eucrefl{\SGvertex_{j}\SGvertex_{k}}\vert_{\ssvertices_{0}}$
	(recall Proposition \ref{prop:SG-symmetry}) and \eqref{eq:p-RF-Diff-homogeneous}.
\item\label{it:SG-p-RF-harmonic-eigenfunc-prf}The invariance of $\form_{p}$
	under $\eucrefl{\SGvertex_{j}\SGvertex_{k}}\vert_{\ssset}$ yields
	$\harprincipal{p}{i}\circ\eucrefl{\SGvertex_{j}\SGvertex_{k}}\vert_{\ssset}=\harprincipal{p}{i}$ and
	$\harsecond{p}{i}\circ\eucrefl{\SGvertex_{j}\SGvertex_{k}}\vert_{\ssset}=-\harsecond{p}{i}$,
	which together with $\harprincipal{p}{i}\circ\ssmap_{i},\harsecond{p}{i}\circ\ssmap_{i}\in\harfunc{\form_{p}}{\ssvertices_{0}}$
	from Proposition \ref{prop:SG-p-RF-harmonic-Vn} implies that
	$\harprincipal{p}{i}\circ\ssmap_{i}=\kappa_{p}\harprincipal{p}{i}$ and $\harsecond{p}{i}\circ\ssmap_{i}=\lambda_{p}\harsecond{p}{i}$
	for some $\kappa_{p},\lambda_{p}\in\mathbb{R}$. The equality
	$\kappa_{p}=\sscndc_{p}^{-1/(p-1)}$ is the counterpart of \cite[Lemma A.1.5]{Kig01}
	and follows in the same way from \ref{it:SG-p-RF-normal-derivative-self-sim}
	with $n=1$, \eqref{eq:p-RF-Diff-homogeneous} and the fact that
	$\form_{p}^{(0)}(\one_{\SGvertex_{j}}+\one_{\SGvertex_{k}};\one_{\SGvertex_{i}})<0$
	by \ref{it:SG-p-RF-normal-derivative-monotone-prf} (or \eqref{eq:p-RF-Diff-homogeneous}).
	Finally,
	$\lambda_{p}=\harsecond{p}{i}(\ssmap_{i}(\SGvertex_{k}))
		<\harprincipal{p}{i}(\ssmap_{i}(\SGvertex_{k}))=\kappa_{p}=\sscndc_{p}^{-1/(p-1)}$
	and $\lambda_{p}=\harsecond{p}{i}(\ssmap_{i}(\SGvertex_{k}))>0$
	by Theorem \ref{thm:SG-p-RF-strong-comp} with, respectively, $U=\ssset\setminus\ssvertices_{0}$
	and $U=\{x\in\ssset\setminus\ssvertices_{0}\mid\lvert x-\SGvertex_{j}\rvert>\lvert x-\SGvertex_{k}\rvert\}$,
	which satisfies $\harsecond{p}{i}\vert_{\partial_{\ssset}U}=\one^{\partial_{\ssset}U}_{\SGvertex_{k}}$
	by $\harsecond{p}{i}\circ\eucrefl{\SGvertex_{j}\SGvertex_{k}}\vert_{\ssset}=-\harsecond{p}{i}$.
\item\label{it:SG-p-RF-harmonic-osc-contraction-prf}By replacing $h$ with
	$h-h(\SGvertex_{i})\one_{\ssset}$ on the basis of \eqref{eq:p-RF-harext-homogeneous}
	we may assume that $h(\SGvertex_{i})=0$, and then this is immediate from
	\eqref{eq:p-RF-harext-homogeneous}, \ref{it:SG-p-RF-harmonic-eigenfunc}
	and Theorem \ref{thm:SG-p-RF-strong-comp}.
\item\label{it:SG-p-RF-harmonic-non-degenerate-prf}This is immediate from
	Proposition \ref{prop:SG-p-RF-harmonic-Vn} and \ref{it:SG-p-RF-harmonic-osc-contraction}.
\qed\end{enumerate}
\end{proof}
\begin{thm}[Perron--Frobenius type theorem]\label{thm:SG-p-RF-Perron-Frobenius}
Let $i\in\ssindex$, $h\in\harfunc{\form_{p}}{\ssvertices_{0}}$ and set
$c_{p,i}(h):=-\sgn\bigl(\form_{p}^{(0)}(h\vert_{\ssvertices_{0}};\one_{\SGvertex_{i}})\bigr)
	\bigl\lvert\form_{p}^{(0)}(h\vert_{\ssvertices_{0}};\one_{\SGvertex_{i}})/\form_{p}^{(0)}(\harprincipal{p}{i}\vert_{\ssvertices_{0}};\one_{\SGvertex_{i}})\bigr\rvert^{1/(p-1)}$.
Then in the norm topologies of both $(\contfunc(\ssset),\lVert\cdot\rVert_{\sup})$
and $(\domain_{p}/\mathbb{R}\one_{\ssset},\form_{p}^{1/p})$,
\begin{equation}\label{eq:SG-p-RF-Perron-Frobenius}
\lim_{n\to\infty}\sscndc_{p}^{n/(p-1)}(h\circ\ssmap_{i^{n}}-h(\SGvertex_{i})\one_{\ssset})
	=c_{p,i}(h)\harprincipal{p}{i}.
\end{equation}
\end{thm}
\begin{proof}
Let $j,k\in\ssindex$ be as in Proposition \ref{prop:SG-p-RF-normal-derivative-self-sim}, and set
$h_{n}:=\sscndc_{p}^{n/(p-1)}(h\circ\ssmap_{i^{n}}-h(\SGvertex_{i})\one_{\ssset})$,
$a_{n}:=h_{n}(\SGvertex_{j})$ and $b_{n}:=h_{n}(\SGvertex_{k})$ for $n\in\mathbb{N}\cup\{0\}$.
Replacing $h$ with $-h$ on the basis of
\eqref{eq:p-RF-harext-homogeneous} if $a_{0}>b_{0}$, we may and do assume that $a_{0}\leq b_{0}$.
Then noting that $\{h_{n}\}_{n\in\mathbb{N}\cup\{0\}}\subset\harfunc{\form_{p}}{\ssvertices_{0}}$
by Proposition \ref{prop:SG-p-RF-harmonic-Vn} and \eqref{eq:p-RF-harext-homogeneous},
we see by an induction on $n$ using \eqref{eq:p-RF-harext-homogeneous},
Proposition \ref{prop:SG-p-RF-normal-derivative-self-sim}-\ref{it:SG-p-RF-harmonic-eigenfunc}
and Proposition \ref{prop:SG-p-RF-comp} that $a_{n-1}\leq a_{n}\leq b_{n}\leq b_{n-1}$
for any $n\in\mathbb{N}$, where $a_{n}\leq b_{n}$ follows from Proposition \ref{prop:SG-p-RF-comp}
with $U=\{x\in\ssset\setminus\ssvertices_{0}\mid\lvert x-\SGvertex_{j}\rvert>\lvert x-\SGvertex_{k}\rvert\}$,
$u=h\circ\ssmap_{i^{n-1}}\circ\eucrefl{\SGvertex_{j}\SGvertex_{k}}\vert_{\ssset}$ and
$v=h\circ\ssmap_{i^{n-1}}$. Thus $a:=\lim_{n\to\infty}a_{n}\in\mathbb{R}$ and
$b:=\lim_{n\to\infty}b_{n}\in\mathbb{R}$ exist, which means that
$\lim_{n\to\infty}\lVert h_{\infty}-h_{n}\rVert_{\sup}=0$ for
$h_{\infty}:=\harext{\form_{p}}{\ssvertices_{0}}[a\one_{\SGvertex_{j}}+b\one_{\SGvertex_{k}}]$ since
$\lVert u-v\rVert_{\sup}=\lVert u\vert_{\ssvertices_{0}}-v\vert_{\ssvertices_{0}}\rVert_{\sup,\ssvertices_{0}}$
for any $u,v\in\harfunc{\form_{p}}{\ssvertices_{0}}$ by \eqref{eq:p-RF-harext-homogeneous}
and Proposition \ref{prop:p-RF-comp}. Then letting $n\to\infty$ in the obvious equality
$h_{n+1}=\sscndc_{p}^{1/(p-1)}h_{n}\circ\ssmap_{i}$ shows that
$h_{\infty}=\sscndc_{p}^{1/(p-1)}h_{\infty}\circ\ssmap_{i}$,
which cannot hold under $a\not=b$ by \eqref{eq:p-RF-harext-homogeneous},
Proposition \ref{prop:SG-p-RF-normal-derivative-self-sim}-\ref{it:SG-p-RF-harmonic-eigenfunc}
and Theorem \ref{thm:SG-p-RF-strong-comp} and thus implies that $a=b$. Therefore
$h_{\infty}=a\harprincipal{p}{i}$ by \eqref{eq:p-RF-harext-homogeneous}, and we further see from
Proposition \ref{prop:SG-p-RF-normal-derivative-self-sim}-\ref{it:SG-p-RF-normal-derivative-self-sim},
\eqref{eq:p-RF-Diff-homogeneous}, $\#\ssvertices_{0}<\infty$ and the continuity of
$\form_{p}^{(0)}(\cdot;\one_{\SGvertex_{i}})$ on $\mathbb{R}^{\ssvertices_{0}}/\mathbb{R}\one_{\ssvertices_{0}}$
from \eqref{eq:p-RF-Diff-Hoelder-cont} that
\begin{equation}\label{eq:SG-p-RF-Perron-Frobenius-coefficient}
\begin{split}
\form_{p}^{(0)}(h\vert_{\ssvertices_{0}};\one_{\SGvertex_{i}})
	=\lim_{n\to\infty}\form_{p}^{(0)}(h_{n}\vert_{\ssvertices_{0}};\one_{\SGvertex_{i}})
	&=\form_{p}^{(0)}(h_{\infty}\vert_{\ssvertices_{0}};\one_{\SGvertex_{i}})\\
=\form_{p}^{(0)}(a\harprincipal{p}{i}\vert_{\ssvertices_{0}};\one_{\SGvertex_{i}})
	&=\sgn(a)\lvert a\rvert^{p-1}\form_{p}^{(0)}(\harprincipal{p}{i}\vert_{\ssvertices_{0}};\one_{\SGvertex_{i}}),
\end{split}
\end{equation}
so that $a=c_{p,i}(h)$ since $\form_{p}^{(0)}(\harprincipal{p}{i}\vert_{\ssvertices_{0}};\one_{\SGvertex_{i}})<0$
by Proposition \ref{prop:SG-p-RF-normal-derivative-self-sim}-\ref{it:SG-p-RF-normal-derivative-monotone-prf}
(or \eqref{eq:p-RF-Diff-homogeneous}). Finally, by
$h_{\infty}\in\harfunc{\form_{p}}{\ssvertices_{0}}$,
$\{h_{n}\}_{n\in\mathbb{N}\cup\{0\}}\subset\harfunc{\form_{p}}{\ssvertices_{0}}$,
$\form_{p}\vert_{\ssvertices_{0}}=\form_{p}^{(0)}$ from Theorem \ref{thm:SG-p-RF},
\ref{it:p-RF1}, $\#\ssvertices_{0}<\infty$ and
$\lim_{n\to\infty}\lVert h_{\infty}\vert_{\ssvertices_{0}}-h_{n}\vert_{\ssvertices_{0}}\rVert_{\sup,\ssvertices_{0}}=0$
we have
\begin{equation}\label{eq:SG-p-RF-Perron-Frobenius-conv-energy}
\form_{p}(h_{n})=\form_{p}^{(0)}(h_{n}\vert_{\ssvertices_{0}})
	\xrightarrow{n\to\infty}\form_{p}^{(0)}(h_{\infty}\vert_{\ssvertices_{0}})
	=\form_{p}(h_{\infty}),
\end{equation}
which together with $\lim_{n\to\infty}\lVert h_{\infty}-h_{n}\rVert_{\sup}=0$ and
Proposition \ref{prop:p-RF-norm-convergence-characterize}-\ref{it:p-RF-norm-convergence-characterize}
yields $\lim_{n\to\infty}\form_{p}(h_{\infty}-h_{n})=0$, completing the proof.
\qed\end{proof}
\begin{rmk}\label{rmk:SG-p-RF-Perron-Frobenius}
The above proof of Theorem \ref{thm:SG-p-RF-Perron-Frobenius} and that of the existence of
$\harprincipal{p}{i}\in\harfunc{\form_{p}}{\ssvertices_{0}}$
with $\harprincipal{p}{i}(\SGvertex_{i})=0$,
$\harprincipal{p}{i}(\SGvertex_{j})\wedge\harprincipal{p}{i}(\SGvertex_{k})>0$ and
$\harprincipal{p}{i}\circ\ssmap_{i}=\sscndc_{p}^{-1/(p-1)}\harprincipal{p}{i}$
in Proposition \ref{prop:SG-p-RF-normal-derivative-self-sim}-\ref{it:SG-p-RF-harmonic-eigenfunc}
rely heavily on the invariance of $\form_{p}$ with respect to
$\eucrefl{\SGvertex_{j}\SGvertex_{k}}\vert_{\ssset}$ and the fact that
$\ssset_{i}\cap(\ssset_{j}\cup\ssset_{k})=\{\ssmap_{i}(\SGvertex_{j}),\ssmap_{i}(\SGvertex_{k})\}$
consists of two points and is invariant with respect to $\eucrefl{\SGvertex_{j}\SGvertex_{k}}\vert_{\ssset}$.
For more general p.-c.f.\ self-similar sets, such very nice geometric features are
usually only partially available or not available, and some alternative arguments are
needed to establish the same kind of results. Fortunately, there is a well-established
theory on Perron--Frobenius type theorems for \emph{nonlinear} $1$-homogeneous
order-preserving maps presented in \cite{LN}, and our strong comparison principle
as in Theorem \ref{thm:SG-p-RF-strong-comp} is strong enough to let us apply
the most relevant results \cite[Corollary 5.4.2 and Theorem 6.5.1]{LN} there.
The statements obtained in this way are similar to
Proposition \ref{prop:SG-p-RF-normal-derivative-self-sim}-\ref{it:SG-p-RF-harmonic-eigenfunc}
and Theorem \ref{thm:SG-p-RF-Perron-Frobenius}, but considerably weaker in the sense
that the counterpart of the eigenfunction $\harprincipal{p}{i}$ is never explicit and
that the convergence result as in \eqref{eq:SG-p-RF-Perron-Frobenius} requires the
additional assumption that $h(x)-h(\SGvertex_{i})\geq 0$ for any
$x\in\ssvertices_{0}\setminus\{\SGvertex_{i}\}$. It is not clear to the authors how negative
$(h-h(\SGvertex_{i})\one_{\ssset})\vert_{\ssvertices_{0}\setminus\{\SGvertex_{i}\}}$
is allowed to be in order for the convergence as in \eqref{eq:SG-p-RF-Perron-Frobenius}
to remain holding.
\end{rmk}

Theorem \ref{thm:SG-p-RF-Perron-Frobenius} allows us to adapt the argument in
\cite[Remark 2.7]{Pei} to conclude the uniqueness of a
$p$-resistance form $\formsecond_{p}^{(0)}$ satisfying
$\renom_{\sscndc_{p},1}(\formsecond_{p}^{(0)})\big\vert_{\ssvertices_{0}}=\formsecond_{p}^{(0)}$.
Note that, as presented in \cite[Remark 2.7 and Corollary 5.7]{Pei},
\emph{the assumption of the $\sssym$-invariance of $\formsecond_{p}^{(0)}$ is NOT needed here}
thanks to the $\sssym$-invariance of $\form_{p}^{(0)}$, $\#\ssvertices_{0}\geq 3$ and the fact that
$\{g\vert_{\ssvertices_{0}}\mid g\in\sssym\}=\{g\mid\textrm{$g\colon\ssvertices_{0}\to\ssvertices_{0}$, $g$ is bijective}\}$.
\begin{thm}\label{thm:SG-p-RF-V0-unique}
Let $\formsecond_{p}^{(0)}$ be a $p$-resistance form on $\ssvertices_{0}$ with the property that
$\renom_{\sscndc_{p},1}(\formsecond_{p}^{(0)})\big\vert_{\ssvertices_{0}}=\formsecond_{p}^{(0)}$.
Then $\formsecond_{p}^{(0)}=c\form_{p}^{(0)}$ for some $c\in(0,\infty)$.
\end{thm}
\begin{prfsktch}
Given the Perron--Frobenius type theorem (Theorem \ref{thm:SG-p-RF-Perron-Frobenius}) above,
this is proved in exactly the same way as \cite[Proof of Theorem 2.6 and Remark 2.7]{Pei}.
\qed\end{prfsktch}

We can also translate Theorem \ref{thm:SG-p-RF-V0-unique} into the uniqueness of
a $p$-resistance form on $\ssset$ with the properties \ref{it:SG-p-RF-SSE1} and
\ref{it:SG-p-RF-SSE2} in Theorem \ref{thm:SG-p-RF}, as follows.
\begin{thm}\label{thm:SG-p-RF-unique}
Let $(\formsecond_{p},\domainsecond_{p})$ be a $p$-resistance form on $\ssset$
satisfying \ref{it:SG-p-RF-SSE1} and \ref{it:SG-p-RF-SSE2}. Then
$\domainsecond_{p}=\domain_{p}$ and $\formsecond_{p}=c\form_{p}$ for some $c\in(0,\infty)$.
\end{thm}
\begin{prfsktch}
Set $\formsecond_{p}^{(0)}:=\formsecond_{p}\vert_{\ssvertices_{0}}$. Then
we can see from \ref{it:SG-p-RF-SSE1} and \ref{it:SG-p-RF-SSE2} that
$\formsecond_{p}\vert_{\ssvertices_{n}}=\renom_{\sscndc_{p},n}(\formsecond_{p}^{(0)})$
for any $n\in\mathbb{N}$, hence
$\renom_{\sscndc_{p},1}(\formsecond_{p}^{(0)})\big\vert_{\ssvertices_{0}}
	=\formsecond_{p}\vert_{\ssvertices_{1}}\vert_{\ssvertices_{0}}=\formsecond_{p}^{(0)}$
by Proposition \ref{prop:p-RF-trace-compatible}, and thus
$\formsecond_{p}^{(0)}=c\form_{p}^{(0)}$ for some $c\in(0,\infty)$
by Theorem \ref{thm:SG-p-RF-V0-unique}. Now it is not difficult to show that
$\domainsecond_{p}=\domain_{p}$ and $\formsecond_{p}=c\form_{p}$,
by the same argument as the proof of Theorem \ref{thm:p-RF-inductive-limit}
(see also \cite[Proof of Theorem 2.36-(2)]{K:Fractal2018} and
\cite[Proof of Lemma 2.3.8]{Kig01}).
\qed\end{prfsktch}
\begin{dfn}[Canonical $p$-resistance form]\label{dfn:SG-cannonical-p-RF}
In view of its uniqueness obtained in Theorem \ref{thm:SG-p-RF-unique},
we call $(\form_{p},\domain_{p})$ as defined in Definition \ref{dfn:SG-p-RF}
the \emph{canonical $p$-resistance form} on the Sierpi\'{n}ski gasket $\ssset$.
\end{dfn}
\begin{rmk}\label{rmk:SG-p-RF-renom-trace-necessary}
Here are a couple of remarks on our choice of the framework of $p$-resistance forms
in relation to the existence and detailed properties of $\form_{p}^{(0)}$.
\begin{enumerate}[label=\textup{(\arabic*)},align=left,leftmargin=*,topsep=4pt,itemsep=2pt]
\item\label{it:SG-p-RF-renom-necessary}As mentioned in the above sketch of the proof
	of Theorem \ref{thm:SG-p-RF-unique}, the existence of a $p$-resistance form
	$\form_{p}^{(0)}$ on $\ssvertices_{0}$ with the property that
	$\renom_{\sscndc_{p},1}(\form_{p}^{(0)})\big\vert_{\ssvertices_{0}}=\form_{p}^{(0)}$
	for some $\sscndc_{p}\in(0,\infty)$ is necessary for that of a $p$-resistance form
	$(\form_{p},\domain_{p})$ on $\ssset$ with the self-similarity \ref{it:SG-p-RF-SSE1}
	and \ref{it:SG-p-RF-SSE2}. In this sense, the construction of $(\form_{p},\domain_{p})$
	based on such $\form_{p}^{(0)}$ as presented in Subsection \ref{ssec:SG-construct-p-RF}
	does not put any restriction on the class of resulting $p$-resistance forms on
	$\ssset$ as far as self-similar ones are concerned. On the other hand, any possible
	proof of the existence of such $\form_{p}^{(0)}$ would inevitably involve
	\emph{the operation of taking traces to subsets, which does NOT preserve the class
	of $p$-energy forms on finite sets of the type \eqref{eq:p-RF-finite-graph}}.
	This is why we are forced to consider some larger class of $p$-energy forms,
	and that of $p$-resistance forms as formulated in Definition \ref{dfn:p-RF}
	seems to be the right one for unifying the existing studies of self-similar
	$p$-energy forms on self-similar sets in \cite{HPS,SW,Kig23,Shi,CGQ}.
\item\label{it:SG-p-RF-trace-necessary}Observe the central roles played
	by the functional $\form_{p}^{(0)}(\cdot;\cdot)$ in the above proofs of
	Theorem \ref{thm:SG-p-RF-strong-comp},
	Proposition \ref{prop:SG-p-RF-normal-derivative-self-sim} and
	Theorem \ref{thm:SG-p-RF-Perron-Frobenius}. These proofs are made possible by
	combining some of the basic properties of $p$-resistance forms given in
	Section \ref{sec:p-RF} with the fact that $\form_{p}^{(0)}$ coincides with the trace
	$\form_{p}\vert_{\ssvertices_{0}}$ of $(\form_{p},\domain_{p})$ to $\ssvertices_{0}$
	and provides a useful discrete characterization \eqref{eq:SG-p-RF-harmonic-Vn-test}
	of the $\form_{p}$-harmonicity. This observation suggests that it is important
	to guarantee nice properties of $\form_{p}^{(0)}$ for further detailed analysis
	of the limit $p$-energy form $(\form_{p},\domain_{p})$, and the framework of
	$p$-resistance forms is helpful for this purpose.%
\end{enumerate}
\end{rmk}
%
\section{$p$-Energy measures and singularity among distinct $p$}\label{sec:SG-p-energy-meas}
%
In this last section, we introduce the $p$-energy measures on the Sierpi\'{n}ski gasket,
present their basic properties, and give the proof that the $p$-energy measures and the
$q$-energy measures are mutually singular for any $p,q\in(1,\infty)$ with $p\not=q$,
which will be proved in \cite{KS:pqEnergySing} in a more general setting of self-similar
$p$-energy forms on p.-c.f.\ self-similar sets with certain very good geometric symmetry.

Throughout this section, we continue to follow the notation in Section \ref{sec:SG-p-RF},
and let $p$ denote an arbitrary element of $(1,\infty)$ unless otherwise stated. First,
the $p$-energy measure $\enermeas{p}{u}$ of $u\in\domain_{p}$ is defined as follows.
Note that, as mentioned in \cite[Proof of Lemma 4-(ii)]{HN}, for $p=2$ the following
definition results in what is known as the $\form_{2}$-energy measure of $u\in\domain_{2}$
in the theory of regular symmetric Dirichlet forms; see \cite[(3.2.14)]{FOT} for the
definition of the latter.
\begin{thm}\label{thm:SG-p-energy-meas}
Let $u\in\domain_{p}$. Then there exists a unique Borel measure $\enermeas{p}{u}$ on $\ssset$
such that $\enermeas{p}{u}(\ssset_{w})=\sscndc_{p}^{\lvert w\rvert}\form_{p}(u\circ\ssmap_{w})$
for any $w\in\words_{*}$. Moreover, $\enermeas{p}{u}(\{x\})=0$ for any $x\in\ssset$.
\end{thm}
\begin{prfsktch}
The uniqueness of $\enermeas{p}{u}$ is immediate from \ref{it:SG-p-RF-SSE2} and the Dynkin
class theorem (see, e.g., \cite[Appendixes, Theorem 4.2]{EK}). To see its existence,
we follow the construction in \cite[Lemma 4.1]{Hin05} (see also \cite[Section 7]{Shi}).
Namely, we consider the (unique, by the Dynkin class theorem) Borel measure
$\enermeasshift{p}{u}$ on $\ssindex^{\mathbb{N}}$ such that
$\enermeasshift{p}{u}(\{w\}\times\ssindex^{\mathbb{N}\cap(\lvert w\rvert,\infty)})
	=\sscndc_{p}^{\lvert w\rvert}\form_{p}(u\circ\ssmap_{w})$
for any $w\in\words_{*}$, which exists by \ref{it:SG-p-RF-SSE2} and
Kolmogorov's extension theorem (see, e.g., \cite[Theorem 12.1.2]{Dud}),
and define a Borel measure $\enermeas{p}{u}$ on $\ssset$ by
$\enermeas{p}{u}:=\enermeasshift{p}{u}\circ\ssproj^{-1}$,
where $\ssproj\colon\ssindex^{\mathbb{N}}\to\ssset$ is the
continuous surjection given by
$\{\ssproj((\omega_{n})_{n\in\mathbb{N}})\}:=\bigcap_{n\in\mathbb{N}}\ssset_{\omega_{1}\ldots\omega_{n}}$
(see, e.g., \cite[Theorem 1.2.3]{Kig01}). Then since
$\ssproj^{-1}(\ssset_{w}\setminus\ssvertices_{*})
	=\bigl(\{w\}\times\ssindex^{\mathbb{N}\cap(\lvert w\rvert,\infty)}\bigr)\setminus\ssproj^{-1}(\ssvertices_{*})$
for any $w\in\words_{*}$ by Proposition \ref{prop:SG-intersecting-cells}-\ref{it:SG-intersecting-cells}
and $\sup_{x\in\ssset}\#\ssproj^{-1}(x)=2<\infty$ by
Proposition \ref{prop:SG-intersecting-cells}-\ref{it:SG-intersecting-cells-words}
(see also \cite[Proof of Lemma 4.2.3]{Kig01}), the desired properties of
$\enermeas{p}{u}$ will follow from the definition of $\enermeasshift{p}{u}$
and the countability of $\ssvertices_{*}$ once we have shown that
\begin{equation}\label{eq:SG-p-energy-meas-shift-atomless}
\enermeasshift{p}{u}(\{\omega\})=0\qquad\textrm{for any $\omega\in\ssindex^{\mathbb{N}}$.}
\end{equation}
Indeed, if $u\in\bigcup_{n\in\mathbb{N}\cup\{0\}}\harfunc{\form_{p}}{\ssvertices_{n}}$,
then \eqref{eq:SG-p-energy-meas-shift-atomless} is not difficult see by combining the
implication from \ref{it:SG-p-RF-harmonic-Vn} to \ref{it:SG-p-RF-harmonic-Vn-self-sim}
in Proposition \ref{prop:SG-p-RF-harmonic-Vn} with
Proposition \ref{prop:SG-p-RF-normal-derivative-self-sim}-\ref{it:SG-p-RF-harmonic-osc-contraction}
(or with \eqref{eq:p-RF-harm-func-Hoelder}, Lemma \ref{lem:SG-p-RF-RM-contraction} and
$\sup_{x,y\in\ssset}\resismet_{\form_{p}}(x,y)^{1/(p-1)}(x,y)<\infty$ from Theorem \ref{thm:SG-p-RF}).
Then \eqref{eq:SG-p-energy-meas-shift-atomless} for $u\in\domain_{p}$ follows from the inequalities
$\enermeasshift{p}{u}(\{\omega\})^{1/p}
	\leq\enermeasshift{p}{v}(\{\omega\})^{1/p}+\enermeasshift{p}{u-v}(\{\omega\})^{1/p}
	\leq\enermeasshift{p}{v}(\{\omega\})^{1/p}+\form_{p}(u-v)^{1/p}$
for $u,v\in\domain_{p}$ implied by \ref{it:p-RF1} and \ref{it:SG-p-RF-SSE2},
Proposition \ref{prop:SG-p-RF-harmonic-approx},
and \eqref{eq:SG-p-energy-meas-shift-atomless} for
$u\in\bigcup_{n\in\mathbb{N}\cup\{0\}}\harfunc{\form_{p}}{\ssvertices_{n}}$.
\qed\end{prfsktch}
\begin{dfn}[$p$-Energy measure]\label{dfn:SG-p-energy-meas}
For each $u\in\domain_{p}$, the Borel measure $\enermeas{p}{u}$ on $\ssset$ as in
Theorem \ref{thm:SG-p-energy-meas} is called the \emph{$\form_{p}$-energy measure} of $u$,
or the \emph{canonical $p$-energy measure} of $u$ on the Sierpi\'{n}ski gasket $\ssset$.
\end{dfn}

We collect some basic properties of the $p$-energy measures in the following propositions
and theorems. The details of these results will be presented in \cite{KS:GCDiff}
under a more general setting of self-similar $p$-energy forms on self-similar sets.
\begin{prop}\label{prop:SG-p-energy-meas-properties-basic}
If $f\colon\ssset\to[0,\infty)$ is Borel measurable and $\lVert f\rVert_{\sup}<\infty$, then:
\begin{enumerate}[label=\textup{(\arabic*)},align=left,leftmargin=*,topsep=4pt,itemsep=2pt]
\item\label{it:SG-p-energy-meas-seminorm}$\bigl(\int_{\ssset}f\,d\enermeas{p}{\cdot}\bigr)^{1/p}$
	is a seminorm on $\domain_{p}$ and $\int_{\ssset}f\,d\enermeas{p}{\one_{\ssset}}=0$;
\item\label{it:SG-p-energy-meas-p-RF5}$\bigl(\int_{\ssset}f\,d\enermeas{p}{\cdot},\domain_{p}\bigr)$
	satisfies \ref{it:p-RF5};
\item\label{it:SG-p-energy-meas-p-RF5-conseq}Proposition \textup{\ref{prop:p-RF5-conseq}}
	with $\int_{\ssset}f\,d\enermeas{p}{\cdot}$ in place of $\form_{p}$ holds.
\end{enumerate}
\end{prop}
\begin{prfsktch}
Since $\{A\times\ssindex^{\mathbb{N}\cap(n,\infty)}\mid\textrm{$n\in\mathbb{N}$, $A\subset\ssindex^{n}$}\}$
is an algebra in $\ssindex^{\mathbb{N}}$ generating $\Borel(\ssindex^{\mathbb{N}})$,
it is not difficult to see from Lemma \ref{lem:p-RF1-p-RF5-cone} and the monotone class theorem
(see, e.g., \cite[Theorem 4.4.2]{Dud}) that $\enermeasshift{p}{\cdot}(A)$ has the properties
stated in \ref{it:SG-p-energy-meas-seminorm} and \ref{it:SG-p-energy-meas-p-RF5}
for any $A\in\Borel(\ssindex^{\mathbb{N}})$, where $\enermeasshift{p}{u}$ denotes
the Borel measure on $\ssindex^{\mathbb{N}}$ defined in the sketch of the proof of
Theorem \ref{thm:SG-p-energy-meas} above. Then \ref{it:SG-p-energy-meas-seminorm} and
\ref{it:SG-p-energy-meas-p-RF5} are immediate by Lemma \ref{lem:p-RF1-p-RF5-cone} and
monotone convergence, and \ref{it:SG-p-energy-meas-p-RF5-conseq} follows from
\ref{it:SG-p-energy-meas-p-RF5} and \ref{it:SG-p-energy-meas-seminorm}
in exactly the same way as Proposition \ref{prop:p-RF5-conseq}.
\qed\end{prfsktch}
\begin{prop}\label{prop:SG-p-energy-meas-Diff}
If $f\colon\ssset\to[0,\infty)$ is Borel measurable and $\lVert f\rVert_{\sup}<\infty$,
then $\int_{\ssset}f\,d\enermeas{p}{\cdot}\colon\domain_{p}/\mathbb{R}\one_{\ssset}\to\mathbb{R}$
is Fr\'{e}chet differentiable on $(\domain_{p}/\mathbb{R}\one_{\ssset},\form_{p}^{1/p})$
and has the same properties as those of $\form_{p}$ in Theorem \textup{\ref{thm:p-RF-Diff}}
with ``$v\not\in\mathbb{R}\one_{\ssset}$'' in \eqref{eq:p-RF-Diff-strictly-convex}
replaced by ``$\int_{\ssset}f\,d\enermeas{p}{v}>0$'' and with the same $c_{p}$,
and for any $u,v\in\domain_{p}$,
\begin{align}
\Borel(\ssset)\ni A\mapsto\enermeas{p}{u;v}(A):=\frac{1}{p}&\frac{d}{dt}\enermeas{p}{u+tv}(A)\bigg\vert_{t=0}
	\mspace{10mu}\textrm{is a Borel signed measure on $\ssset$}\notag\\
\textrm{and}\quad\int_{\ssset}f\,d\enermeas{p}{u;v}=\frac{1}{p}&\frac{d}{dt}\int_{\ssset}f\,d\enermeas{p}{u+tv}\bigg\vert_{t=0}.
\label{eq:SG-p-energy-meas-Diff-dfn}
\end{align}
Further, for any $u,v\in\domain_{p}$ and any Borel measurable functions $f,g\colon\ssset\to[0,\infty]$,
\begin{equation}\label{eq:SG-p-energy-meas-Diff-Hoelder}
\int_{\ssset}fg\,d\bigl\lvert\enermeas{p}{u;v}\bigr\rvert
	\leq\biggl(\int_{\ssset}f^{p/(p-1)}\,d\enermeas{p}{u}\biggr)^{(p-1)/p}\biggl(\int_{\ssset}g^{p}\,d\enermeas{p}{v}\biggr)^{1/p}.
\end{equation}
\end{prop}
\begin{prfsktch}
The first part is proved in exactly the same way as Theorem \ref{thm:p-RF-Diff} on the basis of
Proposition \ref{prop:SG-p-energy-meas-properties-basic}-\ref{it:SG-p-energy-meas-seminorm},\ref{it:SG-p-energy-meas-p-RF5-conseq}.
\eqref{eq:SG-p-energy-meas-Diff-dfn} follows by the finite additivity of
$\frac{1}{p}\frac{d}{dt}\int_{\ssset}f\mspace{2.85mu}d\enermeas{p}{u+tv}\big\vert_{t=0}$ in $f$,
\eqref{eq:p-RF-Diff-Hoelder} for $\int_{\ssset}f\mspace{2.85mu}d\enermeas{p}{\cdot}$ and
dominated convergence. \eqref{eq:SG-p-energy-meas-Diff-Hoelder} can be shown
first for $f=g=\one_{A}$ with $A\in\Borel(\ssset)$ by the definition of
$\bigl\lvert\enermeas{p}{u;v}\bigr\rvert(A)$, \eqref{eq:p-RF-Diff-Hoelder} for
$\int_{\ssset}\one_{B}\,d\enermeas{p}{\cdot}=\enermeas{p}{\cdot}(B)$ with $B\in\Borel(\ssset)$ and
H\"{o}lder's inequality, and then for general $f,g$ by H\"{o}lder's inequality and monotone convergence.
\qed\end{prfsktch}
\begin{thm}[Chain rule for $p$-energy measures]\label{thm:SG-p-energy-meas-chain-rule}
Let $n\in\mathbb{N}$, $u\in\domain_{p}$, $v=(v_{1},\ldots,v_{n})\in\domain_{p}^{n}$,
$\Phi\in\contfunc^{1}(\mathbb{R})$ and $\Psi\in\contfunc^{1}(\mathbb{R}^{n})$.
Then $\Phi(u),\Psi(v)\in\domain_{p}$ and
\begin{equation}\label{eq:SG-p-energy-meas-chain-rule}
d\enermeas{p}{\Phi(u);\Psi(v)}=\sum_{k=1}^{n}\sgn(\Phi'(u))\lvert\Phi'(u)\rvert^{p-1}\partial_{k}\Psi(v)\,d\enermeas{p}{u;v_{k}}.
\end{equation}
\end{thm}
\begin{prfsktch}
\ref{it:p-RF5} implies $\Phi(u),\Psi(v)\in\domain_{p}$. \eqref{eq:SG-p-energy-meas-chain-rule}
on $\ssset_{w}$ for $w\in\words_{*}$ can be proved by approximating its right-hand side
by a Riemann sum associated with the partition $\{\ssset_{w\tau}\}_{\tau\in\words_{n}}$
for $n\in\mathbb{N}$, estimating the difference between the term for $\ssset_{w\tau}$ in the sum and
$\enermeas{p}{\Phi(u);\Psi(v)}(\ssset_{w\tau})=\sscndc_{p}^{\lvert w\tau\rvert}\form_{p}(\Phi(u\circ\ssmap_{w\tau});\Psi(v\circ\ssmap_{w\tau}))$
through \eqref{eq:p-RF-Diff-homogeneous}, \eqref{eq:p-RF-Diff-Hoelder}, \eqref{eq:p-RF-Diff-Hoelder-cont}
and Proposition \ref{prop:p-RF5-conseq}-\ref{it:p-RF-1Lip}, adding up the resulting bounds
over $\tau\in\words_{n}$ via H\"{o}lder's inequality, and then letting $n\to\infty$.
Thus \eqref{eq:SG-p-energy-meas-chain-rule} follows by the Dynkin class theorem
(see, e.g., \cite[Appendixes, Theorem 4.2]{EK}).\hspace*{-3.2pt}%
\qed\end{prfsktch}
\begin{thm}\label{thm:SG-p-energy-meas-image-density}
For any $u\in\domain_{p}$, the Borel measure $\enermeas{p}{u}\circ u^{-1}$ on
$\mathbb{R}$ defined by $\enermeas{p}{u}\circ u^{-1}(A):=\enermeas{p}{u}(u^{-1}(A))$
is absolutely continuous with respect to the Lebesgue measure on $\mathbb{R}$.
\end{thm}
\begin{proof}
This is proved, on the basis of Theorem \ref{thm:SG-p-energy-meas-chain-rule}, in exactly
the same way as \cite[Proposition 7.6]{Shi}, which is a simple adaptation of \cite[Theorem 4.3.8]{CF}.
\qed\end{proof}

The following theorem is an improvement on Theorem \ref{thm:SG-p-RF-regular-str-local}-\ref{it:SG-p-RF-strongly-local}
and gives arguably the strongest possible form of the strong locality of $(\form_{p},\domain_{p})$.
\begin{thm}\label{thm:SG-p-energy-meas-strongly-local}
Let $u_{1},u_{2},v\in\domain_{p}$ and $a,b\in\mathbb{R}$.
Then $\enermeas{p}{u_{1};v}(A)=\enermeas{p}{u_{2};v}(A)$
for any $A\in\Borel(\ssset)$ with $A\subset(u_{1}-u_{2})^{-1}(a)\cup v^{-1}(b)$.
\end{thm}
\begin{proof}
This is proved in the same way as \cite[Theorem 7.7]{Shi} by combining
Theorem \ref{thm:SG-p-energy-meas-image-density} with the finite additivity of
$\enermeas{p}{u_{1};v},\enermeas{p}{u_{2};v}$ and \eqref{eq:p-RF-Diff-Hoelder-cont}
for $\int_{\ssset}\one_{B}\,d\enermeas{p}{\cdot}=\enermeas{p}{\cdot}(B)$ with
$B\in\{A\setminus v^{-1}(b),A\cap v^{-1}(b)\}$ from Proposition \ref{prop:SG-p-energy-meas-Diff}.
\qed\end{proof}

We now turn to our last topic, planned to be treated in \cite{KS:pqEnergySing}, of the
singularity of the $p$-energy measures $\enermeas{p}{u}$ and the $q$-energy measures
$\enermeas{q}{v}$ for $p,q\in(1,\infty)$ with $p<q$. It turns out that this singularity
is implied by the \emph{strict} inequality $\sscndc_{p}^{1/(p-1)}<\sscndc_{q}^{1/(q-1)}$,
which on the other hand may or may not hold depending on the precise geometric nature of
the self-similar set and has been proved in \cite{KS:pqEnergySing} only under the
assumption of certain very good geometric symmetry. In the present setting of the
Sierpi\'{n}ski gasket, these results can be stated as follows.
\begin{thm}\label{thm:SG-p-q-energy-meas-sing}
Let $p,q\in(1,\infty)$ satisfy $p\not=q$. Then $\enermeas{p}{u}\perp\enermeas{q}{v}$
for any $u\in\domain_{p}$ and any $v\in\domain_{q}$.
\end{thm}
\begin{thm}\label{thm:SG-p-RF-resis-scale-decreasing}
The function $(1,\infty)\ni p\mapsto\sscndc_{p}^{1/(p-1)}$ is strictly increasing.
\end{thm}
\begin{rmk}\label{rmk:Viksec-Baudoin-Li}
\emph{Theorems \textup{\ref{thm:SG-p-q-energy-meas-sing}} and \textup{\ref{thm:SG-p-RF-resis-scale-decreasing}}
are NOT always true for more general p.-c.f.\ self-similar sets, even under the assumption
of good geometric symmetry of the set.} Indeed, in the case of the Vicsek set
(Figure \ref{fig:Vicsek} below), Baudoin and Chen have recently observed in
\cite[Definition 2.7, Theorem 3.1 and Remark 3.11]{BC} for any $p\in(1,\infty)$ that
the counterpart of $\sscndc_{p}$ as in Theorem \ref{thm:SG-p-RF-V0-exists} is given by
$3^{p-1}$ and that the $p$-energy measures are absolutely continuous with respect to
the length (one-dimensional Lebesgue) measure on the skeleton of the Vicsek set.
\end{rmk}
%
\begin{figure}[h]\centering
	\includegraphics[height=140pt]{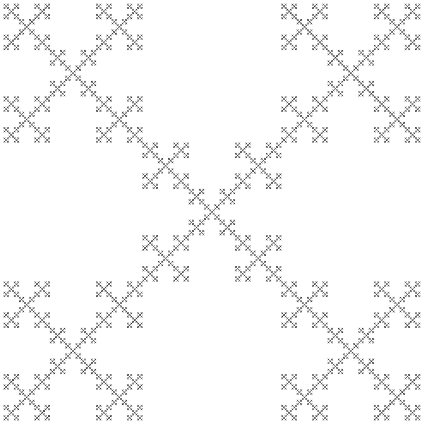}
	\caption{The Vicsek set}
	\label{fig:Vicsek}
\end{figure}

The rest of this section is devoted to the proofs of
Theorems \ref{thm:SG-p-q-energy-meas-sing} and \ref{thm:SG-p-RF-resis-scale-decreasing}.
We begin with the proof of Theorem \ref{thm:SG-p-RF-resis-scale-decreasing}, which goes as follows.
\begin{proof}[\hspace*{-1.45pt}of Theorem \textup{\ref{thm:SG-p-RF-resis-scale-decreasing}}]
Let $p,q\in(1,\infty)$ satisfy $p<q$. To highlight what causes the strict inequality
$\sscndc_{p}^{1/(p-1)}<\sscndc_{q}^{1/(q-1)}$, we first give the proof of the
non-strict one and then describe the necessary modifications to make it strict.

Recall that by Proposition \ref{prop:SG-p-RF-harmonic-Vn} and
Proposition \ref{prop:SG-p-RF-normal-derivative-self-sim}-\ref{it:SG-p-RF-harmonic-osc-contraction},
\begin{equation}\label{eq:SG-p-RF-harmonic-osc-contraction-words}
h\circ\ssmap_{w}\in\harfunc{\form_{p}}{\ssvertices_{0}}
	\qquad\textrm{and}\qquad
	\osc_{\ssset}[h\circ\ssmap_{w}]\leq\sscndc_{p}^{-\lvert w\rvert/(p-1)}\osc_{\ssset}[h]
\end{equation}
for any $h\in\harfunc{\form_{p}}{\ssvertices_{0}}$ and any $w\in\words_{*}$. Note also that
\begin{equation}\label{eq:SG-q-RF-osc-p-RF-comparable}
\bigl(C_{q}^{-1}\form_{q}^{(0)}(u)\bigr)^{1/q}\leq\osc_{\ssvertices_{0}}[u]\leq\bigl(C_{p}\form_{p}^{(0)}(u)\bigr)^{1/p}
	\qquad\textrm{for any $u\in\mathbb{R}^{\ssvertices_{0}}$}
\end{equation}
for some $C_{q},C_{p}\in(0,\infty)$ by $\#\ssvertices_{0}<\infty$,
\hyperlink{p-RF1}{\textup{(RF1)$_{q}$}} for $\form_{q}^{(0)}$ and \ref{it:p-RF1}
for $\form_{p}^{(0)}$. Recalling Theorem \ref{thm:p-RF-trace},
choose any $h\in\harfunc{\form_{p}}{\ssvertices_{0}}$ satisfying
$h\vert_{\ssvertices_{0}}\not\in\mathbb{R}\one_{\ssvertices_{0}}$.
Then for any $n\in\mathbb{N}$, it follows from
\hyperlink{p-RF1}{\textup{(RF1)$_{q}$}} for $\form_{q}^{(0)}$,
Proposition \ref{prop:SG-p-RF-compatible-seq},
\eqref{eq:p-RF-V0-Vn}, \eqref{eq:SG-q-RF-osc-p-RF-comparable}, $q-p>0$,
\eqref{eq:SG-p-RF-harmonic-osc-contraction-words},
$\renom_{\sscndc_{p},n}(\form_{p}^{(0)})=\form_{p}^{(n)}=\form_{p}\vert_{\ssvertices_{n}}$
from Theorem \ref{thm:SG-p-RF}, and  $h\in\harfunc{\form_{p}}{\ssvertices_{0}}\subset\harfunc{\form_{p}}{\ssvertices_{n}}$
that, with $C_{p,q,h}:=C_{q}C_{p}(\osc_{\ssset}[h])^{q-p}\in(0,\infty)$, 
\begin{align}
0<\form_{q}^{(0)}&(h\vert_{\ssvertices_{0}})
	\leq\renom_{\sscndc_{q},n}(\form_{q}^{(0)})(h\vert_{\ssvertices_{n}})
	=\sum_{w\in\words_{n}}\sscndc_{q}^{n}\form_{q}^{(0)}(h\circ\ssmap_{w}\vert_{\ssvertices_{0}})\notag\\
&\leq\sum_{w\in\words_{n}}C_{q}\sscndc_{q}^{n}(\osc_{\ssvertices_{0}}[h\circ\ssmap_{w}\vert_{\ssvertices_{0}}])^{q}\notag\\
&\leq\sum_{w\in\words_{n}}C_{q}C_{p}\sscndc_{q}^{n}\form_{p}^{(0)}(h\circ\ssmap_{w}\vert_{\ssvertices_{0}})(\osc_{\ssvertices_{0}}[h\circ\ssmap_{w}\vert_{\ssvertices_{0}}])^{q-p}\notag\\
&\leq\sum_{w\in\words_{n}}C_{q}C_{p}\sscndc_{q}^{n}\form_{p}^{(0)}(h\circ\ssmap_{w}\vert_{\ssvertices_{0}})\sscndc_{p}^{-n(q-p)/(p-1)}(\osc_{\ssset}[h])^{q-p}\label{eq:SG-p-RF-resis-scale-non-increasing}\\
&=C_{p,q,h}\sscndc_{q}^{n}\sscndc_{p}^{-n(q-1)/(p-1)}\form_{p}^{(n)}(h\vert_{\ssvertices_{n}})
	=C_{p,q,h}\sscndc_{q}^{n}\sscndc_{p}^{-n(q-1)/(p-1)}\form_{p}(h),\notag
\end{align}
whence
$\sscndc_{q}^{1/(q-1)}/\sscndc_{p}^{1/(p-1)}
	\geq\bigl(C_{p,q,h}^{-1}\form_{q}^{(0)}(h\vert_{\ssvertices_{0}})/\form_{p}(h)\bigr)^{1/(n(q-1))}
	\xrightarrow{n\to\infty}1$,
proving $\sscndc_{p}^{1/(p-1)}\leq\sscndc_{q}^{1/(q-1)}$.

To achieve $\sscndc_{p}^{1/(p-1)}<\sscndc_{q}^{1/(q-1)}$, we improve the estimate in the line
\eqref{eq:SG-p-RF-resis-scale-non-increasing} by a compactness argument based on
Proposition \ref{prop:SG-p-RF-normal-derivative-self-sim}-\ref{it:SG-p-RF-harmonic-osc-contraction},\ref{it:SG-p-RF-harmonic-non-degenerate}.
Recall that $\#\ssvertices_{0}<\infty$ and that
$\lVert u-v\rVert_{\sup}=\lVert u\vert_{\ssvertices_{0}}-v\vert_{\ssvertices_{0}}\rVert_{\sup,\ssvertices_{0}}$
and $\osc_{\ssset}[u]=\osc_{\ssvertices_{0}}[u\vert_{\ssvertices_{0}}]$ for any
$u,v\in\harfunc{\form_{p}}{\ssvertices_{0}}$ by \eqref{eq:p-RF-harext-homogeneous}
and Proposition \ref{prop:p-RF-comp}. In particular, for each $i\in\ssindex$ the map
$\mathbb{R}^{\ssvertices_{0}}/\mathbb{R}\one_{\ssvertices_{0}}\ni u\mapsto\harext{\form_{p}}{\ssvertices_{0}}[u]\circ\ssmap_{i}\vert_{\ssvertices_{0}}\in\mathbb{R}^{\ssvertices_{0}}/\mathbb{R}\one_{\ssvertices_{0}}$
can be defined thanks to \eqref{eq:p-RF-harext-homogeneous} and is continuous, and by
considering the maximum of its post-composition by $\osc_{\ssvertices_{0}}[\cdot]$ on the compact set
$\{u\in\mathbb{R}^{\ssvertices_{0}}/\mathbb{R}\one_{\ssvertices_{0}}\mid\textrm{$\osc_{\ssvertices_{0}}[u]=1$, $\osc_{\ssvertices_{0}\setminus\{\SGvertex_{i}\}}[u\vert_{\ssvertices_{0}\setminus\{\SGvertex_{i}\}}]\geq\frac{1}{2}$}\}$,
we see from Proposition \ref{prop:SG-p-RF-normal-derivative-self-sim}-\ref{it:SG-p-RF-harmonic-osc-contraction}
and \eqref{eq:p-RF-harext-homogeneous} that for some $\delta_{p}\in(0,1)$,
\begin{equation}\label{eq:SG-p-RF-harmonic-osc-contraction-strict}
\osc_{\ssset}[h\circ\ssmap_{i}]\leq\delta_{p}\sscndc_{p}^{-1/(p-1)}\osc_{\ssset}[h]
	\quad\textrm{for \emph{some} $i\in\ssindex$ for each $h\in\harfunc{\form_{p}}{\ssvertices_{0}}$.}
\end{equation}
Similarly, since $(\form_{p}^{(0)})^{1/p}$ is a norm on
$\mathbb{R}^{\ssvertices_{0}}/\mathbb{R}\one_{\ssvertices_{0}}$ by \ref{it:p-RF1},
by considering the minimum of
$\mathbb{R}^{\ssvertices_{0}}/\mathbb{R}\one_{\ssvertices_{0}}\ni u\mapsto\form_{p}^{(0)}\bigl(\harext{\form_{p}}{\ssvertices_{0}}[u]\circ\ssmap_{i}\vert_{\ssvertices_{0}}\bigr)\in[0,\infty)$
on $\{u\in\mathbb{R}^{\ssvertices_{0}}/\mathbb{R}\one_{\ssvertices_{0}}\mid\form_{p}^{(0)}(u)=1\}$
for each $i\in\ssindex$, it follows from Proposition \ref{prop:SG-p-RF-normal-derivative-self-sim}-\ref{it:SG-p-RF-harmonic-non-degenerate}
and \eqref{eq:p-RF-harext-homogeneous} that
\begin{equation}\label{eq:SG-p-RF-harmonic-non-degenerate}
\sscndc_{p}\form_{p}^{(0)}(h\circ\ssmap_{i}\vert_{\ssvertices_{0}})\geq\varepsilon_{p}\form_{p}^{(0)}(h\vert_{\ssvertices_{0}})
	\quad\textrm{for any $i\in\ssindex$ and any $h\in\harfunc{\form_{p}}{\ssvertices_{0}}$}
\end{equation}
for some $\varepsilon_{p}\in(0,1)$. Then for any $h\in\harfunc{\form_{p}}{\ssvertices_{0}}$,
combining \eqref{eq:SG-p-RF-harmonic-osc-contraction-words}, \eqref{eq:SG-p-RF-harmonic-osc-contraction-strict}
and \eqref{eq:SG-p-RF-harmonic-non-degenerate} with the fact that
$\sum_{i\in\ssindex}\sscndc_{p}\form_{p}^{(0)}(h\circ\ssmap_{i}\vert_{\ssvertices_{0}})
	=\form_{p}^{(1)}(h\vert_{\ssvertices_{1}})=\form_{p}(h)=\form_{p}^{(0)}(h\vert_{\ssvertices_{0}})$
by \eqref{eq:p-RF-V0-Vn} and $\form_{p}^{(n)}=\form_{p}\vert_{\ssvertices_{n}}$
for $n\in\{0,1\}$ from Theorem \ref{thm:SG-p-RF}, we obtain
\begin{equation}\label{eq:SG-p-RF-resis-scale-decreasing}
\begin{split}
\sum_{i\in\ssindex}\sscndc_{p}&\form_{p}^{(0)}(h\circ\ssmap_{i}\vert_{\ssvertices_{0}})(\osc_{\ssvertices_{0}}[h\circ\ssmap_{i}\vert_{\ssvertices_{0}}])^{q-p}\\
&\leq(1-\varepsilon_{p}(1-\delta_{p}^{q-p}))\sscndc_{p}^{-(q-p)/(p-1)}\form_{p}^{(0)}(h\vert_{\ssvertices_{0}})(\osc_{\ssvertices_{0}}[h\vert_{\ssvertices_{0}}])^{q-p}.
\end{split}
\end{equation}
Now letting $h\in\harfunc{\form_{p}}{\ssvertices_{0}}\setminus\mathbb{R}\one_{\ssset}$
and recalling from \eqref{eq:SG-p-RF-harmonic-osc-contraction-words} that
$h\circ\ssmap_{w}\in\harfunc{\form_{p}}{\ssvertices_{0}}$ for any $w\in\words_{*}$,
we apply \eqref{eq:SG-p-RF-resis-scale-decreasing} in the line before
\eqref{eq:SG-p-RF-resis-scale-non-increasing} and then let $n\to\infty$ to get
$\sscndc_{q}^{1/(q-1)}/\sscndc_{p}^{1/(p-1)}\geq(1-\varepsilon_{p}(1-\delta_{p}^{q-p}))^{-1/(q-1)}>1$,
completing the proof.
\qed\end{proof}

We next turn to the proof of Theorem \ref{thm:SG-p-q-energy-meas-sing}, which is reduced
to the case of $u\in\bigcup_{n\in\mathbb{N}\cup\{0\}}\harfunc{\form_{p}}{\ssvertices_{n}}$
and $v\in\bigcup_{n\in\mathbb{N}\cup\{0\}}\harfunc{\form_{q}}{\ssvertices_{n}}$ by
Proposition \ref{prop:SG-p-RF-harmonic-approx} and the following lemma.
\begin{lem}\label{lem:SG-p-q-energy-meas-sing-cont}
Let $p,q\in(1,\infty)$ satisfy $p\not=q$. If $u\in\domain_{p}$,
$\{u_{n}\}_{n\in\mathbb{N}}\subset\domain_{p}$, $v\in\domain_{q}$ and
$\{v_{n}\}_{n\in\mathbb{N}}\subset\domain_{q}$ satisfy
$\lim_{n\to\infty}\form_{p}(u-u_{n})=0=\lim_{n\to\infty}\form_{q}(v-v_{n})$
and $\enermeas{p}{u_{l}}\perp\enermeas{q}{v_{m}}$ for any $l,m\in\mathbb{N}$,
then $\enermeas{p}{u}\perp\enermeas{q}{v}$.
\end{lem}
\begin{proof}
We follow \cite[Proof of Lemma 3.7-(b)]{KM20}. First, for any $A\in\Borel(\ssset)$,
since $\enermeas{p}{u-u_{n}}(A)\leq\form_{p}(u-u_{n})$ and
$\enermeas{q}{v-v_{n}}(A)\leq\form_{q}(v-v_{n})$ for any $n\in\mathbb{N}$, we see from
Proposition \ref{prop:SG-p-energy-meas-properties-basic}-\ref{it:SG-p-energy-meas-seminorm}
for $\int_{\ssset}\one_{A}\,d\enermeas{p}{\cdot}=\enermeas{p}{\cdot}(A)$ and
$\int_{\ssset}\one_{A}\,d\enermeas{q}{\cdot}=\enermeas{q}{\cdot}(A)$ that
\begin{equation}\label{eq:SG-p-q-energy-meas-convergence}
\lim_{n\to\infty}\enermeas{p}{u_{n}}(A)=\enermeas{p}{u}(A)\qquad\textrm{and}\qquad
	\lim_{n\to\infty}\enermeas{q}{v_{n}}(A)=\enermeas{q}{v}(A).
\end{equation}
For each $l,m\in\mathbb{N}$, by $\enermeas{p}{u_{l}}\perp\enermeas{q}{v_{m}}$
we can choose $A_{l,m}\in\Borel(\ssset)$ so that
$\enermeas{p}{u_{l}}(A_{l,m})=0=\enermeas{q}{v_{m}}(\ssset\setminus A_{l,m})$.
Set $A_{m}:=\bigcap_{l\in\mathbb{N}}A_{l,m}$ for each $m\in\mathbb{N}$ and
$A:=\bigcup_{m\in\mathbb{N}}A_{m}$. Then
$\enermeas{p}{u_{l}}(A_{m})=0=\enermeas{q}{v_{m}}(\ssset\setminus A_{m})$
for any $l,m\in\mathbb{N}$, hence
$\enermeas{p}{u}(A_{m})=\lim_{l\to\infty}\enermeas{p}{u_{l}}(A_{m})=0$
by \eqref{eq:SG-p-q-energy-meas-convergence} and $\enermeas{q}{v_{m}}(\ssset\setminus A)=0$
for any $m\in\mathbb{N}$, and thus $\enermeas{p}{u}(A)=0$ and
$\enermeas{q}{v}(\ssset\setminus A)=\lim_{m\to\infty}\enermeas{q}{v_{m}}(\ssset\setminus A)=0$
by \eqref{eq:SG-p-q-energy-meas-convergence}, proving $\enermeas{p}{u}\perp\enermeas{q}{v}$.
\qed\end{proof}

We prove that $\enermeas{p}{u}\perp\enermeas{q}{v}$ for
$u\in\bigcup_{n\in\mathbb{N}\cup\{0\}}\harfunc{\form_{p}}{\ssvertices_{n}}$ and
$v\in\bigcup_{n\in\mathbb{N}\cup\{0\}}\harfunc{\form_{q}}{\ssvertices_{n}}$, by
combining Theorem \ref{thm:SG-p-RF-resis-scale-decreasing} with the following lemma and theorem.
\begin{lem}\label{lem:SG-p-RF-harmonic-contraction-energy}
There exist $C_{p,1},C_{p,2}\in(0,\infty)$ such that for any
$h\in\harfunc{\form_{p}}{\ssvertices_{0}}$, any $w\in\words_{*}$, any $i\in\ssindex$ with
$h(\SGvertex_{i})\in\{\max_{x\in\ssvertices_{0}}h(x),\min_{x\in\ssvertices_{0}}h(x)\}$
and any $n\in\mathbb{N}\cup\{0\}$,
\begin{align}\label{eq:SG-p-RF-harmonic-contraction-energy-upper}
\form_{p}(h\circ\ssmap_{w})&\leq C_{p,1}\sscndc_{p}^{-\lvert w\rvert p/(p-1)}\form_{p}(h),\\
\form_{p}(h\circ\ssmap_{i^{n}})&\geq C_{p,2}\sscndc_{p}^{-np/(p-1)}\form_{p}(h).
\label{eq:SG-p-RF-harmonic-contraction-energy-lower}
\end{align}
\end{lem}
\begin{proof}
Recall that by $\#\ssvertices_{0}<\infty$ and \ref{it:p-RF1} for $\form_{p}^{(0)}$ we have
\begin{equation}\label{eq:SG-osc-p-RF-comparable}
C_{p}^{-1}(\osc_{\ssvertices_{0}}[u])^{p}\leq\form_{p}^{(0)}(u)
	\leq C_{p}(\osc_{\ssvertices_{0}}[u])^{p}
	\qquad\textrm{for any $u\in\mathbb{R}^{\ssvertices_{0}}$}
\end{equation}
for some $C_{p}\in(0,\infty)$, and that any $h\in\harfunc{\form_{p}}{\ssvertices_{0}}$
satisfies $\form_{p}(h)=\form_{p}^{(0)}(h\vert_{\ssvertices_{0}})$ by
$\form_{p}\vert_{\ssvertices_{0}}=\form_{p}^{(0)}$ from Theorem \ref{thm:SG-p-RF},
$h\circ\ssmap_{w}\in\harfunc{\form_{p}}{\ssvertices_{0}}$
for any $w\in\words_{*}$ by Proposition \ref{prop:SG-p-RF-harmonic-Vn},
and $\osc_{\ssset}[h]=\osc_{\ssvertices_{0}}[h\vert_{\ssvertices_{0}}]$ 
by \eqref{eq:p-RF-harext-homogeneous} and Proposition \ref{prop:p-RF-comp}.
Therefore \eqref{eq:SG-p-RF-harmonic-contraction-energy-upper} holds with
$C_{p,1}=C_{p}^{2}$ by \eqref{eq:SG-osc-p-RF-comparable} and
\eqref{eq:SG-p-RF-harmonic-osc-contraction-words}, and we also see from
\eqref{eq:SG-osc-p-RF-comparable} that \eqref{eq:SG-p-RF-harmonic-contraction-energy-lower}
is implied by the existence of $C_{p,3}\in(0,\infty)$ such that
for any $h\in\harfunc{\form_{p}}{\ssvertices_{0}}$, any $i\in\ssindex$ with
$h(\SGvertex_{i})\in\{\max_{x\in\ssvertices_{0}}h(x),\min_{x\in\ssvertices_{0}}h(x)\}$
and any $n\in\mathbb{N}\cup\{0\}$,
\begin{equation}\label{eq:SG-p-RF-harmonic-contraction-osc-lower}
\osc_{\ssset}[h\circ\ssmap_{i^{n}}]\geq C_{p,3}\sscndc_{p}^{-n/(p-1)}\osc_{\ssset}[h],
\end{equation}
which we prove as follows. Thanks to \eqref{eq:p-RF-harext-homogeneous} we may and do
assume that $h(\SGvertex_{i})=0=\min_{x\in\ssvertices_{0}}h(x)$ by replacing $h$ with
$h-h(\SGvertex_{i})\one_{\ssset}$ or $h(\SGvertex_{i})\one_{\ssset}-h$. Choose
$j\in\ssindex\setminus\{i\}$ so that $h(\SGvertex_{j})=\osc_{\ssset}[h]$ and set
$C_{p,3}:=\sscndc_{p}^{1/(p-1)}\min_{x\in\ssvertices_{0}\setminus\{\SGvertex_{i}\}}\harext{\form_{p}}{\ssvertices_{0}}[\one_{\SGvertex_{j}}]\circ\ssmap_{i}(x)$,
which is independent of $i,j$ by the $\sssym$-invariance of $(\form_{p},\domain_{p})$
from Theorem \ref{thm:SG-p-RF}. Then $C_{p,3}\in(0,1)$ by Theorem \ref{thm:SG-p-RF-strong-comp}
and Proposition \ref{prop:SG-p-RF-normal-derivative-self-sim}-\ref{it:SG-p-RF-harmonic-eigenfunc},
and it follows from Proposition \ref{prop:p-RF-comp}, \eqref{eq:p-RF-harext-homogeneous} and
Proposition \ref{prop:SG-p-RF-normal-derivative-self-sim}-\ref{it:SG-p-RF-harmonic-eigenfunc}
that for any $n\in\mathbb{N}$,
\begin{align*}
&\osc_{\ssset}[h\circ\ssmap_{i^{n}}]
	=\max_{x\in\ssvertices_{0}}h(\ssmap_{i^{n}}(x))\\
&\mspace{10mu}\geq\max_{x\in\ssvertices_{0}}\harext{\form_{p}}{\ssvertices_{0}}\bigl[\osc_{\ssset}[h]\one_{\SGvertex_{j}}\bigr](\ssmap_{i^{n}}(x))
	=\osc_{\ssset}[h]\max_{x\in\ssvertices_{0}}\harext{\form_{p}}{\ssvertices_{0}}[\one_{\SGvertex_{j}}]\circ\ssmap_{i}(\ssmap_{i^{n-1}}(x))\\
&\mspace{10mu}\geq\osc_{\ssset}[h]\max_{x\in\ssvertices_{0}}\harext{\form_{p}}{\ssvertices_{0}}\bigl[C_{p,3}\sscndc_{p}^{-1/(p-1)}\one_{\ssvertices_{0}\setminus\{\SGvertex_{i}\}}\bigr](\ssmap_{i^{n-1}}(x))
	=C_{p,3}\sscndc_{p}^{-n/(p-1)}\osc_{\ssset}[h],
\end{align*}
proving \eqref{eq:SG-p-RF-harmonic-contraction-osc-lower} and thereby
\eqref{eq:SG-p-RF-harmonic-contraction-energy-lower}.
\qed\end{proof}
\begin{thm}[{\cite[Theorem 4.1]{Hin05}}]\label{thm:prob-meas-filtration-singular}
Let $(\Omega,\probspsigmaalg,\mathbb{P})$ be a probability space and
$\{\probspsigmaalg_{n}\}_{n\in\mathbb{N}\cup\{0\}}$ a non-decreasing sequence of $\sigma$-algebras
in $\Omega$ such that $\bigcup_{n\in\mathbb{N}\cup\{0\}}\probspsigmaalg_{n}$ generates $\probspsigmaalg$.
Let $\widetilde{\mathbb{P}}$ be a probability measure on $(\Omega,\probspsigmaalg)$ such that
$\widetilde{\mathbb{P}}\vert_{\probspsigmaalg_{n}}\ll\mathbb{P}\vert_{\probspsigmaalg_{n}}$
for any $n\in\mathbb{N}\cup\{0\}$, and for each $n\in\mathbb{N}$ define
$\alpha_{n}\in L^{1}(\Omega,\probspsigmaalg_{n},\mathbb{P}\vert_{\probspsigmaalg_{n}})$ by
\begin{equation}\label{eq:prob-meas-filtration-singular-alphan}
\alpha_{n}:=
	\begin{cases}
	\dfrac{d(\widetilde{\mathbb{P}}\vert_{\probspsigmaalg_{n}})/d(\mathbb{P}\vert_{\probspsigmaalg_{n}})}{d(\widetilde{\mathbb{P}}\vert_{\probspsigmaalg_{n-1}})/d(\mathbb{P}\vert_{\probspsigmaalg_{n-1}})}
		& \textrm{on $\{d(\widetilde{\mathbb{P}}\vert_{\probspsigmaalg_{n-1}})/d(\mathbb{P}\vert_{\probspsigmaalg_{n-1}})>0\}$,} \\
	0 & \textrm{on $\{d(\widetilde{\mathbb{P}}\vert_{\probspsigmaalg_{n-1}})/d(\mathbb{P}\vert_{\probspsigmaalg_{n-1}})=0\}$,}
	\end{cases}
\end{equation}
so that $\mathbb{E}[\sqrt{\alpha_{n}}\mid\probspsigmaalg_{n-1}]\leq 1$
$\mathbb{P}\vert_{\probspsigmaalg_{n-1}}$-a.s.\ by conditional Jensen's inequality,
where $\mathbb{E}[\cdot\mid\probspsigmaalg_{n-1}]$ denotes the conditional
expectation given $\probspsigmaalg_{n-1}$ with respect to $\mathbb{P}$. If
\begin{equation}\label{eq:prob-meas-filtration-singular}
\sum_{n\in\mathbb{N}}(1-\mathbb{E}[\sqrt{\alpha_{n}}\mid\probspsigmaalg_{n-1}])=\infty
	\qquad\textrm{$\mathbb{P}$-a.s.,}
\end{equation}
then $\widetilde{\mathbb{P}}\perp\mathbb{P}$.
\end{thm}

We will apply Theorem \ref{thm:prob-meas-filtration-singular} under the setting of the following lemma.
\begin{lem}\label{lem:prob-meas-filtration-singular-SG-N}
Set $\Omega:=\ssset$, $\probspsigmaalg:=\mathscr{B}(\ssset)$ and let
$\mathbb{P},\widetilde{\mathbb{P}}$ be probability measures on $(\Omega,\probspsigmaalg)$
such that $\mathbb{P}(\ssset_{w})>0$ for any $w\in\words_{*}$ and
$\mathbb{P}(\ssvertices_{*})=\widetilde{\mathbb{P}}(\ssvertices_{*})=0$.
Let $N\in\mathbb{N}$, set
$\probspsigmaalg_{n}:=\{A\cup\bigcup_{w\in\Lambda}(\ssset_{w}\setminus\ssvertices_{nN})
	\mid\textrm{$\Lambda\subset\words_{nN}$, $A\subset\ssvertices_{nN}$}\}$
for each $n\in\mathbb{N}\cup\{0\}$, so that $\{\probspsigmaalg_{n}\}_{n\in\mathbb{N}\cup\{0\}}$
is a non-decreasing sequence of $\sigma$-algebras in $\Omega$ by
Proposition \textup{\ref{prop:SG-intersecting-cells}-\ref{it:SG-intersecting-cells}},
$\bigcup_{n\in\mathbb{N}\cup\{0\}}\probspsigmaalg_{n}$ generates $\probspsigmaalg$, and
$\widetilde{\mathbb{P}}\vert_{\probspsigmaalg_{n}}\ll\mathbb{P}\vert_{\probspsigmaalg_{n}}$
for any $n\in\mathbb{N}\cup\{0\}$. Let $n\in\mathbb{N}$
and define $\alpha_{n}\in L^{1}(\Omega,\probspsigmaalg_{n},\mathbb{P}\vert_{\probspsigmaalg_{n}})$
by \eqref{eq:prob-meas-filtration-singular-alphan}. Then for each $w\in\words_{(n-1)N}$,
\begin{equation}\label{eq:prob-meas-filtration-singular-SG-N}
\mathbb{E}[\sqrt{\alpha_{n}}\mid\probspsigmaalg_{n-1}]\vert_{\ssset_{w}\setminus\ssvertices_{(n-1)N}}=
	\begin{cases}
	\displaystyle\sum_{\tau\in\words_{N}}\sqrt{\frac{\widetilde{\mathbb{P}}(\ssset_{w\tau})}{\widetilde{\mathbb{P}}(\ssset_{w})}}
		\sqrt{\frac{\mathbb{P}(\ssset_{w\tau})}{\mathbb{P}(\ssset_{w})}} & \textrm{if $\widetilde{\mathbb{P}}(\ssset_{w})>0$,} \\
	0 & \textrm{if $\widetilde{\mathbb{P}}(\ssset_{w})=0$.}
	\end{cases}
\end{equation}
\end{lem}
\begin{proof}
This follows easily by direct calculations based on \eqref{eq:prob-meas-filtration-singular-alphan}
and Proposition \ref{prop:SG-intersecting-cells}-\ref{it:SG-intersecting-cells}.
\qed\end{proof}
\begin{proof}[\hspace*{-1.45pt}of Theorem \textup{\ref{thm:SG-p-q-energy-meas-sing}}]
We first prove $\enermeas{p}{u}\perp\enermeas{q}{v}$ for $u\in\harfunc{\form_{p}}{\ssvertices_{0}}$ and
$v\in\harfunc{\form_{q}}{\ssvertices_{0}}$. Without loss of generality we may and do assume $p<q$, let
$C_{p,2}\in(0,\infty)$ be as in \eqref{eq:SG-p-RF-harmonic-contraction-energy-lower},
$C_{q,1}\in(0,\infty)$ be as in \eqref{eq:SG-p-RF-harmonic-contraction-energy-upper}
with $q$ in place of $p$, and choose $N\in\mathbb{N}$ so that
$C_{q,1}\sscndc_{q}^{-N/(q-1)}\leq\frac{1}{2}C_{p,2}\sscndc_{p}^{-N/(p-1)}$,
which is possible since $\sscndc_{q}^{-1/(q-1)}<\sscndc_{p}^{-1/(p-1)}$
by Theorem \ref{thm:SG-p-RF-resis-scale-decreasing}. Noting that we clearly have
$\enermeas{p}{u}\perp\enermeas{q}{v}$ if $\enermeas{p}{u}(\ssset)\enermeas{q}{v}(\ssset)=0$,
we assume that $\enermeas{p}{u}(\ssset)\enermeas{q}{v}(\ssset)>0$, set
$(\Omega,\probspsigmaalg,\mathbb{P}):=\bigl(\ssset,\Borel(\ssset),\enermeas{p}{u}(\ssset)^{-1}\enermeas{p}{u}\bigr)$,
let $\{\probspsigmaalg_{n}\}_{n\in\mathbb{N}\cup\{0\}}$ denote the non-decreasing sequence
of $\sigma$-algebras in $\Omega$ with $\bigcup_{n\in\mathbb{N}\cup\{0\}}\probspsigmaalg_{n}$
generating $\probspsigmaalg$ as defined in Lemma \ref{lem:prob-meas-filtration-singular-SG-N},
and set $\widetilde{\mathbb{P}}:=\enermeas{q}{v}(\ssset)^{-1}\enermeas{q}{v}$, so that
$\mathbb{P}(\ssset_{w})\widetilde{\mathbb{P}}(\ssset_{w})>0$ for any $w\in\words_{*}$ by
Proposition \ref{prop:SG-p-RF-normal-derivative-self-sim}-\ref{it:SG-p-RF-harmonic-non-degenerate},
\ref{it:p-RF1} for $(\form_{p},\domain_{p})$ and \hyperlink{p-RF1}{\textup{(RF1)$_{q}$}} for
$(\form_{q},\domain_{q})$, and $\mathbb{P}(\ssvertices_{*})=0=\widetilde{\mathbb{P}}(\ssvertices_{*})$
by Theorem \ref{thm:SG-p-energy-meas} and the countability of $\ssvertices_{*}$. In particular,
$\widetilde{\mathbb{P}}\vert_{\probspsigmaalg_{n}}\ll\mathbb{P}\vert_{\probspsigmaalg_{n}}$
for any $n\in\mathbb{N}\cup\{0\}$, and define
$\alpha_{n}\in L^{1}(\Omega,\probspsigmaalg_{n},\mathbb{P}\vert_{\probspsigmaalg_{n}})$
by \eqref{eq:prob-meas-filtration-singular-alphan} for each $n\in\mathbb{N}$.
We claim that for any $n\in\mathbb{N}$,
\begin{equation}\label{eq:prob-meas-filtration-singular-SG-N-upper-bound}
\begin{split}
&\lVert\mathbb{E}[\sqrt{\alpha_{n}}\mid\probspsigmaalg_{n-1}]\rVert_{\sup,\ssset\setminus\ssvertices_{(n-1)N}}\\
&\leq\max\bigl\{\sqrt{st}+\sqrt{(1-s)(1-t)}\bigm\vert\textrm{$s,t\in[0,1]$, $2s\leq C_{p,2}\sscndc_{p}^{-N/(p-1)}\leq t$}\bigr\}\\
&=:\delta_{p,N}\in(0,1).
\end{split}
\end{equation}
Indeed, let $n\in\mathbb{N}$ and $w\in\words_{(n-1)N}$. Then since
$\form_{p}(u\circ\ssmap_{w})\form_{q}(v\circ\ssmap_{w})>0$ by
$\mathbb{P}(\ssset_{w})\widetilde{\mathbb{P}}(\ssset_{w})>0$ and
$u\circ\ssmap_{w}\in\harfunc{\form_{p}}{\ssvertices_{0}}$ and
$v\circ\ssmap_{w}\in\harfunc{\form_{q}}{\ssvertices_{0}}$
by Proposition \ref{prop:SG-p-RF-harmonic-Vn}, taking $i\in\ssindex$ such that
$u\circ\ssmap_{w}(\SGvertex_{i})=\min_{x\in\ssvertices_{0}}u\circ\ssmap_{w}(x)$,
we see from \eqref{eq:SG-p-RF-harmonic-contraction-energy-upper} with $q$ in place of $p$,
the way above of having chosen $N$ and \eqref{eq:SG-p-RF-harmonic-contraction-energy-lower} that
\begin{align}
\frac{\widetilde{\mathbb{P}}(\ssset_{wi^{N}})}{\widetilde{\mathbb{P}}(\ssset_{w})}
	&=\frac{\enermeas{q}{v}(\ssset_{wi^{N}})}{\enermeas{q}{v}(\ssset_{w})}
	=\frac{\sscndc_{q}^{N}\form_{q}(v\circ\ssmap_{wi^{N}})}{\form_{q}(v\circ\ssmap_{w})}
	\leq C_{q,1}\sscndc_{q}^{-N/(q-1)}\leq\frac{1}{2}C_{p,2}\sscndc_{p}^{-N/(p-1)},\notag\\
\frac{\mathbb{P}(\ssset_{wi^{N}})}{\mathbb{P}(\ssset_{w})}
	&=\frac{\enermeas{p}{u}(\ssset_{wi^{N}})}{\enermeas{p}{u}(\ssset_{w})}
	=\frac{\sscndc_{p}^{N}\form_{p}(u\circ\ssmap_{wi^{N}})}{\form_{p}(u\circ\ssmap_{w})}
	\geq C_{p,2}\sscndc_{p}^{-N/(p-1)},
\label{eq:prob-meas-filtration-singular-SG-N-non-parallel}
\end{align}
and hence from Lemma \ref{lem:prob-meas-filtration-singular-SG-N}, the Cauchy--Schwarz
inequality, $\widetilde{\mathbb{P}}(\ssvertices_{*})=0=\mathbb{P}(\ssvertices_{*})$,
Proposition \ref{prop:SG-intersecting-cells}-\ref{it:SG-intersecting-cells}
and \eqref{eq:prob-meas-filtration-singular-SG-N-non-parallel} that
\begin{equation}\label{eq:prob-meas-filtration-singular-SG-N-upper-bound-prf}
\begin{split}
\mathbb{E}&[\sqrt{\alpha_{n}}\mid\probspsigmaalg_{n-1}]\vert_{\ssset_{w}\setminus\ssvertices_{(n-1)N}}\\
&=\sqrt{\frac{\widetilde{\mathbb{P}}(\ssset_{wi^{N}})}{\widetilde{\mathbb{P}}(\ssset_{w})}}
	\sqrt{\frac{\mathbb{P}(\ssset_{wi^{N}})}{\mathbb{P}(\ssset_{w})}}
	+\sum_{\tau\in\words_{N}\setminus\{i^{N}\}}\sqrt{\frac{\widetilde{\mathbb{P}}(\ssset_{w\tau})}{\widetilde{\mathbb{P}}(\ssset_{w})}}
		\sqrt{\frac{\mathbb{P}(\ssset_{w\tau})}{\mathbb{P}(\ssset_{w})}}\\
&\leq\sqrt{\frac{\widetilde{\mathbb{P}}(\ssset_{wi^{N}})}{\widetilde{\mathbb{P}}(\ssset_{w})}}
	\sqrt{\frac{\mathbb{P}(\ssset_{wi^{N}})}{\mathbb{P}(\ssset_{w})}}
	+\sqrt{1-\frac{\widetilde{\mathbb{P}}(\ssset_{wi^{N}})}{\widetilde{\mathbb{P}}(\ssset_{w})}}
		\sqrt{1-\frac{\mathbb{P}(\ssset_{wi^{N}})}{\mathbb{P}(\ssset_{w})}}
	\leq\delta_{p,N},
\end{split}
\end{equation}
proving \eqref{eq:prob-meas-filtration-singular-SG-N-upper-bound}. Thus
$\sum_{n\in\mathbb{N}}(1-\mathbb{E}[\sqrt{\alpha_{n}}\mid\probspsigmaalg_{n-1}](x))=\infty$
for any $x\in\ssset\setminus\ssvertices_{*}$ by \eqref{eq:prob-meas-filtration-singular-SG-N-upper-bound}
and in particular for $\mathbb{P}$-a.e.\ $x\in\ssset=\Omega$ by $\mathbb{P}(\ssvertices_{*})=0$,
so that Theorem \ref{thm:prob-meas-filtration-singular} is applicable and yields
$\widetilde{\mathbb{P}}\perp\mathbb{P}$, namely $\enermeas{p}{u}\perp\enermeas{q}{v}$.

Next, let $u\in\bigcup_{n\in\mathbb{N}\cup\{0\}}\harfunc{\form_{p}}{\ssvertices_{n}}$
and $v\in\bigcup_{n\in\mathbb{N}\cup\{0\}}\harfunc{\form_{q}}{\ssvertices_{n}}$. Then
choosing $n\in\mathbb{N}\cup\{0\}$ so that $u\in\harfunc{\form_{p}}{\ssvertices_{n}}$
and $v\in\harfunc{\form_{q}}{\ssvertices_{n}}$, for each $w\in\words_{n}$ we have
$\enermeas{p}{u}\circ\ssmap_{w}=\sscndc_{p}^{n}\enermeas{p}{u\circ\ssmap_{w}}$ and
$\enermeas{q}{v}\circ\ssmap_{w}=\sscndc_{q}^{n}\enermeas{q}{v\circ\ssmap_{w}}$
by the uniqueness of $\enermeas{p}{u\circ\ssmap_{w}}$ and $\enermeas{q}{v\circ\ssmap_{w}}$
from Theorem \ref{thm:SG-p-energy-meas}, $u\circ\ssmap_{w}\in\harfunc{\form_{p}}{\ssvertices_{0}}$
and $v\circ\ssmap_{w}\in\harfunc{\form_{q}}{\ssvertices_{0}}$
by Proposition \ref{prop:SG-p-RF-harmonic-Vn}, and hence
$\enermeas{p}{u}(\ssmap_{w}(A_{w}))=\sscndc_{p}^{n}\enermeas{p}{u\circ\ssmap_{w}}(A_{w})=0$ and
$\enermeas{q}{v}(\ssset_{w}\setminus\ssmap_{w}(A_{w}))=\sscndc_{q}^{n}\enermeas{q}{v\circ\ssmap_{w}}(\ssset\setminus A_{w})=0$
for some $A_{w}\in\Borel(\ssset)$ by the previous paragraph. Thus
$\enermeas{p}{u}\bigl(\bigcup_{w\in\words_{n}}\ssmap_{w}(A_{w})\bigr)=0
	=\enermeas{q}{v}\bigl(\ssset\setminus\bigcup_{w\in\words_{n}}\ssmap_{w}(A_{w})\bigr)$,
proving $\enermeas{p}{u}\perp\enermeas{q}{v}$.

Finally, combining the result of the last paragraph with
Proposition \ref{prop:SG-p-RF-harmonic-approx} and Lemma \ref{lem:SG-p-q-energy-meas-sing-cont},
we conclude that $\enermeas{p}{u}\perp\enermeas{q}{v}$ for any $u\in\domain_{p}$ and
any $v\in\domain_{q}$.
\qed\end{proof}

\begin{thebibliography}{99}
	
	\bibitem{AF} R.\ Adams and J.\ J.\ F.\ Fournier,
	\emph{Sobolev Spaces}, Second edition,
	Pure Appl.\ Math.\ (Amst.), vol.\ 140,
	Elsevier/Academic Press, Amsterdam, 2003. \mr{2424078}
	
	
	\bibitem{Bar98} M.\ T.\ Barlow, Diffusions on fractals, in:
	\emph{Lectures on Probability Theory and Statistics (Saint-Flour, 1995)},
	Lecture Notes in Math., vol.\ 1690, Springer-Verlag, Berlin, 1998, pp.\ 1--121. \mr{1668115}
	
	
	
	
	
	
	
	
	
	
	
	\bibitem{BC} F.\ Baudoin and L.\ Chen,
	Sobolev spaces and Poincar\'{e} inequalities on the Vicsek fractal,
	\emph{Ann.\ Fenn.\ Math.}\ \textbf{48} (2023), no.\ 1, 3--26. \mr{4494039}
	
	
	
	\bibitem{BH} N.\ Bouleau and F.\ Hirsch,
	\emph{Dirichlet Forms and Analysis on Wiener Space},
	de Gruyter Studies in Mathematics, vol.\ 14,
	Walter de Gruyter \& Co., Berlin, 1991. \mr{1133391}
	
	
	
	\bibitem{Cla} B.\ Claus,
	Energy spaces, Dirichlet forms and capacities in a nonlinear setting,
	\emph{Potential Anal.}\ \textbf{58} (2023), no.\ 1, 159--179. \mr{4535921}
	
	\bibitem{CCK} S.\ Cao, Z.-Q.\ Chen and T.\ Kumagai,
	\emph{On Kigami's conjecture of the embedding $\mathcal{W}^{p}(K)\subset C(K)$},
	preprint, 2023. \arxiv{2307.10449}
	
	\bibitem{CGQ} S.\ Cao, Q.\ Gu, and H.\ Qiu,
	$p$-energies on p.c.f. self-similar sets,
	\emph{Adv.\ Math.}\ \textbf{405} (2022), 108517. \mr{4437617}
	
	\bibitem{CF} Z.-Q.\ Chen and M.\ Fukushima,
	\emph{Symmetric Markov Processes, Time Change, and Boundary Theory},
	London Mathematical Society Monographs Series, vol.\ 35,
	Princeton University Press, Princeton, NJ, 2012. \mr{2849840}
	
	\bibitem{Dud} R.\ M.\ Dudley,
	\emph{Real Analysis and Probability}, Revised reprint of the 1989 original,
	Cambridge Stud.\ Adv.\ Math., vol.\ 74, Cambridge University Press,
	Cambridge, 2002. \mr{1932358}
	
	\bibitem{EK} S.\ N.\ Ethier and T.\ G.\ Kurtz,
	\emph{Markov Processes: Characterization and Convergence},
	Wiley Series in Probability and Mathematical Statistics,
	John Wiley \& Sons, Inc., New York, 1986. \mr{0838085}
	
	
	\bibitem{FOT} M.\ Fukushima, Y.\ Oshima, and M.\ Takeda,
	\emph{Dirichlet Forms and Symmetric Markov Processes},
	Second revised and extended edition, de Gruyter Studies in Mathematics, vol.\ 19,
	Walter de Gruyter \& Co., Berlin, 2011. \mr{2778606}
	
	
	
	
	
	
	
	
	\bibitem{HPS} P.\ E.\ Herman, R.\ Peirone and R.\ S.\ Strichartz,
	$p$-energy and $p$-harmonic functions on Sierpinski gasket type fractals,
	\emph{Potential Anal.}\ \textbf{20} (2004), no.\ 2, 125--148. \mr{2032945}
	
	\bibitem{Hin05} M.\ Hino,
	On singularity of energy measures on self-similar sets,
	\emph{Probab.\ Theory Related Fields} \textbf{132} (2005), no.\ 2, 265--290. \mr{2199293}
	
	
	\bibitem{HN} M.\ Hino and K.\ Nakahara,
	On singularity of energy measures on self-similar sets. II,
	\emph{Bull.\ London Math.\ Soc.}\ \textbf{38} (2006), no.\ 6, 1019--1032. \mr{2285256}
	
	
	\bibitem{K:Fractal2018} N.\ Kajino,
	\emph{Introduction to Laplacians and heat equations on fractals} (in Japansese),
	notes for an intensive lecture course at Faculty of Science, Nara Women's University, 2018.
	\url{https://www.kurims.kyoto-u.ac.jp/~nkajino/lectures/2018/Fractal2018.html}
	
	
	\bibitem{KM20} N.\ Kajino and M.\ Murugan,
	On singularity of energy measures for symmetric diffusions with full off-diagonal heat
	kernel estimates, \emph{Ann.\ Probab.}\ \textbf{48} (2020), no.\ 6, 2920--2951. \mr{4164457}
	
	\bibitem{KM23} N.\ Kajino and M.\ Murugan,
	On the conformal walk dimension: quasisymmetric uniformization for symmetric diffusions,
	\emph{Invent.\ math.}\ \textbf{231} (2023), 263--405. \mr{4526824}
	
	\bibitem{KS:GCDiff} N.\ Kajino and R.\ Shimizu,
	\emph{Contraction properties and differentiability of $p$-energy forms
	with applications to nonlinear potential theory on self-similar sets},
	in preparation.
	
	\bibitem{KS:StrComp} N.\ Kajino and R.\ Shimizu,
	\emph{Strong comparison principle for $p$-harmonic functions and uniqueness
	of symmetry-invariant $p$-energy forms on p.-c.f.\ self-similar sets},
	in preparation.
	
	\bibitem{KS:pqEnergySing} N.\ Kajino and R.\ Shimizu,
	\emph{On singularity of $p$-energy measures among distinct values of $p$
	for p.-c.f.\ self-similar sets}, in preparation.
	
	\bibitem{Kig01} J.\ Kigami,
	\emph{Analysis on Fractals}, Cambridge Tracts in Math., vol.\ 143,
	Cambridge University Press, Cambridge, 2001. \mr{1840042}
	
	\bibitem{Kig12} J.\ Kigami,
	Resistance forms, quasisymmetric maps and heat kernel estimates.
	\emph{Mem.\ Amer.\ Math.\ Soc.}\ \textbf{216} (2012), no.\ 1015. \mr{2919892}
	
	\bibitem{Kig23} J.\ Kigami,
	Conductive homogeneity of compact metric spaces and construction of $p$-energy,
	\emph{Mem.\ Eur.\ Math.\ Soc.}\ \textbf{5} (2023). \mr{4615714}
	
	
	
	
	
	\bibitem{LN} B.\ Lemmens and R.\ Nussbaum,
	\emph{Nonlinear Perron-Frobenius Theory},
	Cambridge Tracts in Math., vol.\ 189,
	Cambridge University Press, Cambridge, 2012. \mr{2953648}
	
	
	\bibitem{MR} Z.-M.\ Ma and M.\ R\"{o}ckner,
	\emph{Introduction to the Theory of (Non-Symmetric) Dirichlet Forms},
	Universitext, Springer-Verlag, Berlin, 1992. \mr{1214375}
	
	
	\bibitem{MS} M.\ Murugan and R.\ Shimizu,
	\emph{First-order Sobolev spaces, self-similar energies and energy measures on the
	Sierpi\'{n}ski carpet}, preprint, 2023. \arxiv{2308.06232}
	
	\bibitem{Pei} R.\ Peirone,
	Convergence and uniqueness problems for Dirichlet forms on fractals,
	\emph{Boll.\ Unione Mat.\ Ital.\ (8)} \textbf{3-B} (2000), no.\ 2, 431--460. \mr{1769995}
	
	
	
	\bibitem{Shi} R.\ Shimizu, Construction of $p$-energy and associated
	energy measures on Sierpi\'{n}ski carpets, \emph{Trans.\ Amer.\ Math.\ Soc.}, in press.
	\doi{10.1090/tran/9036}
	
	\bibitem{Str} R.\ S.\ Strichartz,
	\emph{Differential Equations on Fractals: A Tutorial},
	Princeton Univ.\ Press, Princeton, NJ, 2006. \mr{2246975}
	
	\bibitem{SW} R.\ S.\ Strichartz and C.\ Wong,
	The $p$-Laplacian on the Sierpinski gasket,
	\emph{Nonlinearity} \textbf{17} (2004), no.\ 2, 595--616. \mr{2039061}
	
	
	
	
	
\end{thebibliography}
\end{document}